\pdfoutput=1
\documentclass{article}
\usepackage{arxiv-preprint}

\usepackage{shuffle}
\usepackage{placeins}

\newcommand{\graft}{\curvearrowright}
\newcommand{\ins}{\triangleright}

\newcommand{\ad}{\textrm{ad}}
\newcommand{\plus}{\diamond}
\newcommand{\minus}{\setminus}
\newcommand{\one}{\mathbf{1}}
\newcommand{\dF}{\mathbb{F}}
\newcommand{\blank}{{-}}
\newcommand{\id}{\mathrm{id}}

\newcommand{\R}{\mathbb{R}}
\newcommand{\N}{\mathbb{N}}

\newcommand{\MM}{\mathcal{M}}
\newcommand{\OO}{\mathcal{O}}

\newcommand{\UU}{\mathcal{U}}
\newcommand{\LL}{\mathcal{L}}

\newtheorem{theorem}{Theorem}[section]
\newtheorem{definition}[theorem]{Definition}
\newtheorem*{definition*}{Definition}

\newtheorem{proposition}[theorem]{Proposition}

\newtheorem{remark}[theorem]{Remark}
\newtheorem*{remark*}{Remark}

\newtheorem*{remarks*}{Remarks}
\newtheorem{corollary}[theorem]{Corollary}

\newtheorem*{notation*}{Notation}

\newtheorem*{example*}{Example}

\newtheorem*{examples*}{Examples}

\numberwithin{equation}{section}

\begin{document}

\title{Convergence analysis of generalized modified splitting methods using multi-index series}
\titlemark{Convergence analysis of generalized modified splitting methods using multi-index series}

\emsauthor{1}{
	\givenname{Eugen}
	\surname{Bronasco}
	\orcid{0009-0006-3904-5523}
}{E.~Bronasco}
%%%% Repeat the same fields for each numbered author
\emsauthor{2}{
	\givenname{Mechthild}
	\surname{Thalhammer}
	\orcid{0000-0001-5958-6219}
}{M.~Thalhammer}

%%%% Please provide detailed address info for each author
%%%% Use the same numbering as for \emsauthor above
%%%% Please look up the ROR ID of your institute here: https://ror.org
\Emsaffil{1}{
	\department{Department of Mathematical Sciences}
	\organisation{Chalmers and Gothenburg University}
	\zip{412 58}
	\city{Göteborg}
	\country{Sweden}
	\affemail{bronasco@chalmers.se}}
%%%% Repeat the same fields for each numbered author
%%%% If some author has multiple affiliations, repeat the fields for each affiliation
%%%% Number the affiliations using {}
\Emsaffil{2}{
	\department{Department of Mathematics}
	\organisation{University of Innsbruck}
	\zip{6020}
	\city{Innsbruck}
	\country{Austria}
	\affemail{mechthild.thalhammer@uibk.ac.at}}

%------
% Add MSC 2020 codes according to https://zbmath.org/classification/.
% A unique primary MSC code (in curly brackets) is mandatory,
% while secondary MSC codes (in square brackets) are optional.
%------
\classification[65M12,47Axx]{65L20}

%------
% Add a list of keywords.
%------

%------
% Insert your abstract.
%------

\begin{abstract}
    We consider splitting methods for partial differential equations involving unbounded operators.
    For non-time-reversible dynamics, such as dissipative systems, negative splitting coefficients are generally not admissible because they require stepping backward in time, leading to an order barrier when all coefficients are required to be positive.
    We introduce generalized modified splitting methods that overcome this algebraic order barrier while retaining nonnegative splitting coefficients. To analyze their convergence, we develop a decorated multi-index series formalism and establish a substitution law that provides a systematic framework for deriving order conditions. This law also yields an order-raising procedure based on canceling the leading error through a modification of a single stage. We also derive an explicit composition law for extended multi-index series allowing the composition of methods with independently scaled timesteps common in composition and processed methods. We illustrate the construction with two sixth-order methods that use different sets of commutator corrections. Python scripts automate the computation and verification of order conditions and the construction of the methods. Finally, we discuss connections with word series, Lie--Butcher series, and multi-index Butcher series.
\end{abstract}

%%%%%%%%%%%%%%%%%%%%%%%%%%%%%%%%%%
%%%%%%%%%%%%%%%%%%%%%%%%%%%%%%%%%%

\maketitle

\section{Introduction}

A broad class of ordinary and partial differential equations can be decomposed into a sum of terms whose corresponding evolution problems are significantly easier to integrate separately than the original equation. Splitting integrators are numerical methods specifically designed to exploit such decompositions by approximating the full evolution through a sequence of substeps associated with the individual terms. In this work, we consider the linear differential equation
\begin{equation}
    \label{eq:ode}
    \frac{d}{dt} u(t) = A u(t) + B u(t) \,, \quad 0 \leq t \leq T \,,
\end{equation}
where $A$ and $B$ are linear operators. Let $[A, B] = AB - BA$ denote the commutator and
\[ \ad^j_A (B) := [A, \ad^{j-1}_A (B)] \,, \quad \ad^0_A (B) = B \,. \]
We consider splitting integrators $\Phi_h$ with timestep $h$ of the form
\begin{equation}
    \label{eq:splitting}
    u_{n+1} = \Phi_h (u_n) = \prod_{j=1}^s e^{b_{s-j+1} h B} e^{a_{s-j+1} h A} u_n \,,
\end{equation}
fully characterized by the coefficients $a_j, b_j \in \R$, $j=1,\ldots,s$.

An integrator $\Phi_h$ is of \emph{convergence order} $p$ if
\[ | u(t_n) - u_n | = \OO(h^p) \,, \]
where $t_n = nh$ for $n \in \N$. \emph{Order $p$ conditions} are algebraic constraints on $a_j, b_j \in \R$, $j=1,\dots,s$, which guarantee that the corresponding integrator is of order $p$ under appropriate regularity assumptions. Order conditions for splitting methods can be obtained by using the Baker--Campbell--Hausdorff formula as
\[ e^{hA} e^{hB} = e^{hA + hB + \widetilde{BCH}(hA,hB)} \,, \]
where $\widetilde{BCH}(hA, hB)$ is a formal series of nested commutators of $A$ and $B$. Applying this formula successively to \eqref{eq:splitting} leads to the derivation of order conditions by requiring the commutators produced by $\widetilde{BCH}$ to vanish.

We build on the alternative approach of \cite{ThalhammerHEO08,BlanesNFS13}, which compares the coefficients in expansions of the exact and numerical solutions while keeping $e^{hA}$ unexpanded. The resulting multi-index order conditions for splitting methods provide the starting point of this paper.

It was shown in \cite{BlanesNNC05,ShengSLP89,SuzukiGTF91} that achieving order $4$ of convergence requires some of the coefficients $a_j$ to be negative and some of the coefficients $b_j$ to be negative as well. For non-time-reversible systems, this difficulty can be addressed either by considering complex-valued $a_j, b_j$ \cite{CastellaSMC09,HansenHOS09}, or by considering splitting methods with modified potentials \cite{ruth1983canonical,SuzukiGTF91,SuzukiHEP95} of the following form:
\begin{equation}
    \label{eq:splitting_mod}
    u_{n+1} = \Phi_h (u_n) = \prod_{j=1}^s e^{b_{s-j+1} h B + \hat{c}_{s-j+1} h^3 [B,[A,B]]} e^{a_{s-j+1} h A} u_n \,,
\end{equation}
fully characterized by the coefficients $a_j, b_j, \hat{c}_j \in \R$, $j=1,\ldots,s$. The addition of the nested commutator $[B, [A, B]]$ introduces additional degrees of freedom which allow us to achieve order $4$ convergence with positive coefficients. Convergence order analysis for splitting methods with modified potentials was performed in \cite{KieriSCF15}.

However, \cite{AuzingerNGS19} demonstrates that splitting methods with modified potentials of order $6$ must have some of the coefficients $a_j$ negative. Moreover, Theorem~1.18 of \cite{BeauchardCTS26} identifies a hierarchy of commutators of degrees $3,5,\dots,2n-1$ that permits order $2n$ with nonnegative coefficients $a_j$, and proves that this is the maximal order assuming nonnegative $a_j$. Therefore, we consider a natural generalization which we call \emph{generalized modified splitting methods} of the form
\begin{equation}
    \label{eq:splitting_gen}
    u_{n+1} = \Phi_h (u_n) = \prod_{j=1}^s e^{b_{s-j+1} h B_{s-j+1}} e^{a_{s-j+1} h A} u_n \,,
\end{equation}
where the operator $B_j$ for $j=1,\dots,s$ is given by
\begin{align*}
    B_j = B &+ c^{(j)}_1 h [A,B] + c^{(j)}_2 h^2 [A,[A,B]] \\ 
            &+ c^{(j)}_3 h^2 [B,[A,B]] + c^{(j)}_4 h^3 [A,[A,[A,B]]] + \cdots \,,
\end{align*}
where we choose to write the elements of the Lie algebra in the appropriate Hall basis.

The error analysis of generalized modified splitting methods becomes increasingly involved as the order and the number of stages grow. Each stage-dependent modification contributes to the error expansion of the full composition. We develop a systematic coefficient framework for organizing these contributions while retaining the evolution $e^{hA}$ and expressing the remaining terms through nested commutators. This form of the expansion is adapted to commutator-based regularity estimates for evolution equations involving unbounded operators. Our focus is on methods with nonnegative principal coefficients that overcome the second- and fourth-order barriers of standard and modified splitting methods.

Our first contribution is the composition law for extended multi-index series (Theorem~\ref{thm:composition}). It uses a semidirect product to give explicit coefficients for compositions with independently scaled timesteps, allowing future analysis of composition and processed splitting methods using the formalism introduced in this paper.

The main ingredient is a substitution law (Theorem~\ref{thm:substitution}) that computes how corrections within individual stages change the coefficients of the full method. It accommodates general nested commutators and provides the coefficient calculations needed for order conditions and local error expansions. Building on the analysis of standard and modified splitting methods, we state a convergence estimate for linear equations with unbounded $A$ and bounded $B$ under commutator regularity assumptions in Section~\ref{SectionConvergence}. The same section presents numerical experiments for the two sixth-order constructions and comparisons with methods from the literature.

The substitution law yields a constructive order-raising procedure (Theorem~\ref{thm:modified_eq}): the leading error can be canceled by modifying one stage, without changing the coefficients $a_j,b_j$ or adding stages. Under the stated symmetry assumptions, a central-stage correction raises the order by two. We illustrate this procedure with two sixth-order methods, both retaining nonnegative coefficients $a_j,b_j$. The first extends Chin's fourth-order modified potential method (Section~\ref{sec:high_order_chin}); the second uses Gauss--Legendre nodes to reduce the number of commutator corrections, at the cost of additional $A$-substeps (Section~\ref{sec:high_order_gauss}).

To support these constructions, we provide Python scripts implementing the substitution law, order verification, and the illustrated order-raising procedure. We also discuss relations with other series formalisms.
For nonlinear evolution equations, each nonlinear part of the equation can be represented by a linear Lie derivative operator. This suggests extending the coefficient calculus developed here to these operators in future work.

The remainder of the paper is organized as follows. Section~\ref{sec:intro_multi_index} recalls multi-index series and the order conditions for ordinary splitting methods, establishing the framework used throughout the paper. Section~\ref{sec:composition} develops composition of extended multi-index series to allow the composition of methods with different timesteps common in splitting and processed methods. Section~\ref{sec:substitution} introduces decorated expansions and proves the stage-dependent substitution law which results in a multi-index series expansion of generalized modified splitting methods and the corresponding derivation of order conditions. Section~\ref{sec:high_order_construction} demonstrates the constructive use of this law through single-stage order raising, two sixth-order examples, and their automatic computation. Section~\ref{SectionConvergence} connects the formal order analysis with convergence under regularity assumptions and numerical tests of order and stability. Section~\ref{sec:conclusions} summarizes the results and discusses further directions. The appendices collect the explicit order-six conditions and the coefficients of the comparison methods to support verification and reproducibility.

\subsection{Generalized modified splitting methods in the literature}
\label{sec:literature}

The sixth-order force-gradient methods of Omelyan, Mryglod, and Folk~\cite{OmelyanCHF02} use the correction $[B,[A,B]]$ and involve negative splitting coefficients. Our constructions retain nonnegative $a_j,b_j$ by allowing additional commutators, but do not resolve the difficulty of evaluating the resulting modified propagators discussed in~\cite[Section II.B]{OmelyanCHF02}.

Higher commutators are already used in the processed Runge--Kutta--Nystr\"om methods of Blanes, Casas, and Ros~\cite{BlanesHRK01}. Unlike those constructions, our formalism does not require identities such as $[B,[B,[B,A]]]=0$ or explicitly computable modified flows. We also avoid processing by compositions of splitting flows: the processor and its inverse have opposite substep coefficients, introducing backward steps that are unsuitable for the dissipative equations considered here.

Symmetric Zassenhaus splittings instead introduce separate exponentials of commutators. Our formalism, in particular the substitution law, can be applied to this arrangement by allowing $B_j=c_0^{(j)}B+\cdots$: stages containing only commutators have $c_0^{(j)}=0$, with intervening coefficients $a_j=0$ where needed. This algebraic extension is distinct from the local error analysis of Bergold, Iserles, and Kropielnicka~\cite{BergoldSZS26}, which permits both $A$ and $B$ to be unbounded; our convergence theorem in Section~\ref{SectionConvergence} assumes bounded $B$.

The coefficients of the literature methods used in the numerical comparisons are collected in Appendix~\ref{app:literature_methods}.

\section{Multi-index series}
\label{sec:intro_multi_index}

We first fix a series notation for the multi-index expansions of the exact solution and ordinary splitting methods from \cite{ThalhammerHEO08,BlanesNFS13}.
Throughout the paper, $\N=\{0,1,2,\dots\}$. A \emph{multi-index} is a tuple of natural numbers of arbitrary length. A multi-index $\mu$ is written as $\mu=(\mu_1,\dots,\mu_k)$, and the empty multi-index is denoted by $\one$. The set of all multi-indices is $\N^\infty$. For $\mu\in\N^k$, define its factorial and size by
\[ \mu! := \prod_{j=1}^k\mu_j!, \qquad |\mu| := k+\sum_{j=1}^k\mu_j, \]
with $\one!=1$ and $|\one|=0$. To each multi-index we associate the operator
\[ \dF(\mu):=h^{|\mu|}\ad_{-A}^{\mu_k}(B)\cdots\ad_{-A}^{\mu_1}(B), \qquad \dF(\one):=\id. \]

\begin{definition}
    \label{def:butcher}
    A \emph{multi-index series} is the formal series
    \[ B(\alpha):=\sum_{\mu\in\N^\infty}\frac{\alpha(\mu)}{\mu!}\dF(\mu), \]
    where $\alpha:\N^\infty\to\R$ is a coefficient map.
\end{definition}

\begin{remark}[Relation to other series representations]\label{rem:series-representations}
Classical B-series index elementary differentials by rooted trees, whereas word series index iterated derivatives by words over an alphabet; see~\cite[Sections 2--3]{Sanz-SernaFSN15}. Our multi-indices are ordered words over $\N$, with each letter recording the number of commutators with $-A$ acting on $B$. Their elementary operators are therefore ordered products of commutators, rather than the tree-indexed elementary differentials of classical B-series. We retain the factor $e^{hA}$ instead of expanding the entire evolution in powers of $A$. This organization is useful for commutator-based error estimates.

The term \emph{multi-index} also has a different meaning in the multi-indice B-series of Bruned, Ebrahimi-Fard, and Hou~\cite{BrunedMB25}: there it denotes an exponent vector recording vertex-arity multiplicities in a commutative monomial. Here the order of the entries is essential: in general, $(i,j)$ and $(j,i)$ represent different elementary operators.
\end{remark}

We use this notation below for the exact solution and for splitting methods. The exact-solution expansion and the order conditions are collected in the next two subsections; both are established in the cited literature.

\subsection{Exact solution}
\label{sec:butcher_exact}

The exact-solution expansion used here was first obtained in
\cite{ThalhammerCAH12,ThalhammerHEO08} and recast in multi-index form in
\cite{BlanesNFS13}. We recall the result without reproducing its derivation.

\begin{proposition}
    \label{prop:butcher_exact}
    The exact solution of \eqref{eq:ode} can be written as a multi-index series:
    \[ u(h) = e^{hA} B(1/\gamma) u_n \,, \quad \text{with } u(0) = u_n \,, \]
    where $\gamma(\mu) := \prod_{l=1}^k\left(l+\sum_{j=1}^l\mu_j\right)$.
\end{proposition}
For the ordinary splitting method \eqref{eq:splitting}, we use $\alpha$ to denote the coefficient map in its corresponding multi-index expansion; the explicit formula is recalled in Corollary \ref{cor:splitting_expansion}.

\subsection{Order conditions}
\label{sec:order_conditions}

Since both the exact solution and the splitting method \eqref{eq:splitting} can be written using multi-index series, the order conditions for an arbitrary order $p$ can be obtained by comparing the corresponding coefficients up to multi-indices $\mu$ of size $|\mu| \leq p$. That is, a splitting method \eqref{eq:splitting} is of order $p$ if 
\[ \alpha (\mu) = \frac1{\gamma(\mu)} \,, \quad \text{for all } \mu \in \N^\infty \,, \ |\mu| \leq p \,, \]
under appropriate regularity assumptions.
This recovers the order conditions obtained in \cite{BlanesNFS13,ThalhammerCAH12,ThalhammerHEO08}.

In \cite{KieriSCF15}, order conditions are derived for splitting methods with modified potentials \eqref{eq:splitting_mod} where the operator $B_j$ in \eqref{eq:splitting_gen} equals $B + h^2 c_j [B, [A,B]]$. We generalize this result in Section~\ref{sec:substitution} where we consider a general $B_j$ given by
\[ B_j = B + c^{(j)}_1 h [A,B] + c^{(j)}_2 h^2 [A,[A,B]] + c^{(j)}_3 h^2 [B,[A,B]] + c^{(j)}_4 h^3 [A,[A,[A,B]]] + \cdots \,. \]

\subsection{Lyndon multi-indices}
\label{sec:lyndon}

Let $\MM$ be the vector space spanned by the multi-indices $\N^\infty$.
Let $\cdot : \MM \otimes \MM \to \MM$ denote the concatenation product defined by
\[ (\mu_1, \dots, \mu_k) \cdot (\tilde\mu_1, \dots, \tilde\mu_{\tilde{k}}) = (\mu_1, \dots, \mu_k, \tilde\mu_1, \dots, \tilde\mu_{\tilde{k}}) \,. \]
Let $\shuffle : \MM \otimes \MM \to \MM$ denote the shuffle product which interleaves the indices in the left and right operand while preserving their relative order. The empty multi-index $\one$ is the neutral element. It is defined recursively for $\one \neq \mu \in \N^k $ and $\one \neq \tilde\mu \in \N^{\tilde{k}}$ as
\[ \mu \shuffle \tilde\mu = \mu_{1:1} \big( \mu_{2:k} \shuffle \tilde\mu \big) + \tilde{\mu}_{1:1} \big( \mu \shuffle \tilde{\mu}_{2:\tilde{k}} \big) \,, \]
where $\mu_{i:i + j} := (\mu_i, \dots, \mu_{i+j})$ for $\mu = (\mu_1, \dots, \mu_k)$.
For example,
\begin{align*}
    (1,2,3) \shuffle (4,5) = (1,2,3,4,5) &+ (1,2,4,3,5) + (1,2,4,5,3) + (1,4,2,3,5) \\
                                         &+ (1,4,2,5,3) + (1,4,5,2,3) + (4,1,2,3,5) \\
                                         &+ (4,1,2,5,3) + (4,1,5,2,3) + (4,5,1,2,3) \,.
\end{align*}
The number of terms in $\mu \shuffle \tilde\mu$ is given by the binomial coefficient $\frac{(k + \tilde{k})!}{k! \tilde{k}!}$ where $\mu \in \N^k$ and $\tilde\mu \in \N^{\tilde{k}}$.

We extend coefficient maps linearly to $\MM$. A key property of $\alpha$ and $1/\gamma$ which allows a considerable reduction in the number of order conditions is given by
\[ \alpha(\mu \shuffle \tilde\mu) = \alpha(\mu) \alpha(\tilde\mu) \quad \text{and} \quad \frac1{\gamma}(\mu \shuffle \tilde\mu) = \frac1{\gamma}(\mu) \frac1{\gamma}(\tilde\mu) \,. \]
That is, both $\alpha$ and $1/\gamma$ are \emph{characters} with respect to the shuffle product. This property was used for order condition reduction in \cite{BlanesNFS13}. The vector space $\MM$ together with the shuffle product $\shuffle$ forms the shuffle algebra which is generated by Lyndon multi-indices.

Assume the usual dictionary order on the set $\N^\infty$ of multi-indices. A cyclic permutation $\sigma$ is a permutation that moves the number at the $i$-th position of $\mu \in \N^k$ to the position $(i+l) \mod k$ for some $l \in \N$, that is, $\sigma(\mu)_i = \mu_{(i-l) \mod k}$, where the position numbering starts at $0$. For example, the cyclic permutations of $(1,2,3)$ are $(3,1,2)$ for $l=1$ and $(2,3,1)$ for $l=2$.

\begin{definition}
    \label{def:lyndon}
    A multi-index $\mu$ is a \emph{Lyndon multi-index} if it is strictly smaller than any of its non-trivial cyclic permutations.
\end{definition}

The set of Lyndon multi-indices is denoted by $\N^\infty_L$. For example, all Lyndon multi-indices up to size $5$ are listed below
\begin{align*}
    \N^\infty_L = \{ (0) \,, \quad &(1) \,, \quad (0,1) \,, \quad (2) \,, \quad (0,0,1) \,, \quad (0,2) \,, \quad (3) \,, \\
                          &(0,0,0,1) \,, \quad (0,0,2) \,, \quad (0,1,1) \,, \quad (0,3) \,, \quad (1,2) \,, \quad (4) \,, \quad \dots \} \,.
\end{align*}
Reduction of order conditions implies that a splitting method is of order $p$ if
\[ \alpha (\mu) = \frac1{\gamma(\mu)} \,, \quad \text{for all } \mu \in \N_L^\infty \,, \ |\mu| \leq p \,, \]
under appropriate regularity assumptions.
The character identities above are standard for the multi-index expansion and are proved in \cite{BlanesNFS13}; we use them without reproducing the proof.

\section{Composition of extended multi-index series}
\label{sec:composition}

Given two splitting methods $\Phi_h$ and $\Psi_h$, their composition is again a splitting method and thus can be expanded as a multi-index series. In this section, we derive the coefficient map for the composition of $\Psi_{\rho_2 h}$ and $\Phi_{\rho_1 h}$ with scaled time-steps in terms of the coefficient maps of $\Phi_h$ and $\Psi_h$.

Let $\R \times \N^\infty$ be the set of \emph{extended multi-indices}, that is, pairs $(\rho, \mu)$ with $\rho \in \R$ and $\mu \in \N^\infty$.
Extended multi-index series are defined as,
\[ B(\rho, \alpha) := \sum_{\mu \in \N^\infty} \frac{\alpha (\mu)}{\mu!} \dF(\rho, \mu) \,, \]
where $\dF(\rho, \mu) = e^{\rho h A} \dF(\mu)$. Note that $B(\rho, \alpha) = e^{\rho hA} B(\alpha)$.

\begin{proposition}
    \label{prop:extended_series}
    Consider a splitting method $\Phi_h = e^{hA} B(\alpha)$, then $\Phi_{\rho h} = B(\rho, \alpha_\rho)$ where $\alpha_\rho (\mu) := \rho^{|\mu|} \alpha(\mu)$.
\end{proposition}
\begin{proof}
    The statement follows from the following computation,
    \begin{align*}
        \Phi_{\rho h} &= \sum_{\mu \in \N^\infty} \frac{\alpha (\mu)}{\mu!} e^{\rho h A} (\rho h)^{|\mu|} \ad^{\mu_k}_{-A} (B) \cdots \ad^{\mu_1}_{-A} (B) \\
                      &= \sum_{\mu \in \N^\infty} \frac{\alpha_\rho (\mu)}{\mu!} e^{\rho h A} \dF(\mu) = \sum_{\mu \in \N^\infty} \frac{\alpha_\rho (\mu)}{\mu!} \dF(\rho, \mu) = B(\rho, \alpha_\rho) \,,
    \end{align*}
    where $\mu \in \N^k$ and we use the definition of $\dF(\mu)$, $\dF(\rho, \mu)$ and $\alpha_\rho$.
\end{proof}

Let $d : \MM \to \MM$ be defined, for $(i) \in \N^1$ and $\mu, \eta \in \N^\infty$, by
\[ d\one = 0 \,, \quad d(i) := (i+1) \,, \quad d(\mu \cdot \eta) := (d\mu) \cdot \eta + \mu \cdot (d\eta) \,, \]
that is, $d$ is a derivation with respect to the concatenation product.
Let $\MM_n$ be the span of the multi-indices of size $n$ and let $\widehat{\MM}=\prod_{n\geq0}\MM_n$ be its degree completion. Let $\rho \in \R$ and consider the map $E_\rho : \widehat{\MM} \to \widehat{\MM}$ defined by
\[ E_\rho (\mu) := \exp(\rho d) \mu \,. \]
Note that $E$ is the action of the group $(\R, +)$ on $\widehat{\MM}$ and satisfies, for $\rho_1, \rho_2 \in \R$ and $\mu, \eta \in \N^\infty$,
\begin{equation}
    \label{eq:E_props}
    E_{\rho_1} \circ E_{\rho_2} = E_{\rho_1 + \rho_2} \,, \quad E_\rho (\mu \cdot \eta) = E_\rho (\mu) \cdot E_\rho (\eta) \,.
\end{equation}
Let $\MM_\R$ be the vector space spanned by the extended multi-indices $\R \times \N^\infty$ and let $\widehat{\MM}_\R$ be its degree completion with the degree given by the size of the multi-index component. Consider the semidirect product $\graft : \widehat{\MM}_\R \otimes \widehat{\MM}_\R \to \widehat{\MM}_\R$ given by
\[ (\rho_1, \eta) \graft (\rho_2, \mu) = (\rho_1 + \rho_2, \eta \cdot E_{\rho_1} \mu) \,, \]
with $\rho_1, \rho_2 \in \R$ and $\eta, \mu \in \N^\infty$. Note that we take the elements of $\MM_\R$ to be linear in the second component, but not in the first.
\iffalse
\begin{equation}
    \label{eq:graft}
    \eta \graft_\rho \mu := \sum_{p \in \N^{\#\mu}} \frac{\rho^{\sum_{i=1}^{\#\mu} p_i}}{p!} \eta \cdot (p \plus \mu) \,, \quad \text{with } \eta \graft_\rho \one = \eta
\end{equation}
\fi
The product $\graft$ is associative following the properties \eqref{eq:E_props} and it arises in the composition of two operators $e^{\rho_2 hA} \dF (\mu)$ and $e^{\rho_1 hA} \dF(\eta)$ as demonstrated in Proposition~\ref{prop:graft}.

For the algebraic arguments, regard $A$ and $B$ as formal generators and let $\LL$ be the Lie subalgebra of the free Lie algebra on $A,B$ generated by $\ad_{-A}^j(B)$, $j\in\N$. To handle products of these elements, as they occur in $\dF(\mu)$, we work in the universal enveloping algebra
\[ \UU(\LL)=T(\LL)\big/\langle x\otimes y-y\otimes x-[x,y]:x,y\in\LL\rangle. \]
Here $T(\LL)$ is the tensor algebra of $\LL$, and the denominator is the two-sided ideal imposing the relation $xy-yx=[x,y]$. Thus the Lie bracket is represented by the commutator in this associative algebra. The action $\ad_{-A} : \LL \to \LL$ extends uniquely to a derivation of $\UU(\LL)$ satisfying the Leibniz rule
\[ \ad_{-A}(xy)=\ad_{-A}(x)y+x\ad_{-A}(y). \]
This extension is used in the proof of Proposition~\ref{prop:graft}.

\begin{proposition}
    \label{prop:graft}
    For $\mu, \eta \in \N^\infty$ and $\rho_1, \rho_2 \in \R$, the composition of operators $\dF(\rho_2, \mu)$ and $\dF(\rho_1, \eta)$ is given by
    \[ \dF(\rho_2, \mu) \cdot \dF(\rho_1, \eta) = \dF\big( (\rho_1, \eta) \graft (\rho_2, \mu) \big) \,. \]
\end{proposition}
\begin{proof}
    Given $\rho \in \R, \mu \in \N^k$, using $\dF(d\mu) = h \ad_{-A} \dF(\mu)$ we obtain the identity
    \begin{align*}
        \dF(0, \mu) \cdot \dF(\rho, \one) &= h^{|\mu|} \prod_{j=1}^k \ad_{-A}^{\mu_{k-j+1}} (B) e^{\rho hA} \\
                                          &= h^{|\mu|} e^{\rho hA} \prod_{j=1}^k \exp(\ad_{-\rho h A}) \big( \ad_{-A}^{\mu_{k-j+1}} (B) \big) = \dF(\rho, E_\rho \mu) \,.
    \end{align*}
    Therefore, we have
    \begin{align*}
        \dF(\rho_2, \mu) \cdot \dF(\rho_1, \eta) &= \dF(\rho_2, \one) \cdot \dF(0, \mu) \cdot \dF(\rho_1, \one) \cdot \dF(0, \eta) \\
                                                 &= \dF(\rho_2, \one) \cdot \dF(\rho_1, E_{\rho_1} \mu) \cdot \dF(0, \eta) = \dF(\rho_1 + \rho_2, \eta \cdot E_{\rho_1} \mu) \,,
    \end{align*}
    where we use the definition of $\dF$ over extended multi-indices to establish the equalities.
\end{proof}

Associate each coefficient map $\alpha$ with the formal sum $\delta(\alpha)$ given by
\[ \delta(\alpha):=\sum_{\mu\in\N^\infty}\frac{\alpha(\mu)}{\mu!}\mu \,, \]
then, $B(\rho, \alpha)=\dF\big(\rho, \delta(\alpha)\big)$. Consider the pairing $\langle\blank,\blank\rangle_! : \widehat\MM \times \MM \to \R$ defined by
\[ \langle\mu,\nu\rangle_! = \begin{cases}\mu! & \mu=\nu,\\0 & \mu\neq\nu.\end{cases} \]
Then $\langle\delta(\alpha),\mu\rangle_!=\alpha(\mu)$. Let $d^\dagger : \MM \to \MM$ be the adjoint of $d$ with respect to the pairing, that is, $\langle d^\dagger \mu, \nu\rangle_! = \langle \mu, d\nu\rangle_!$. Let $(i) \in \N^1$ with $i \geq 1$ and $\mu, \eta \in \N^\infty$, the map $d^\dagger$ is given by
\[ d^\dagger \one = d^\dagger (0) = 0 \,, \quad d^\dagger(i) = i (i - 1) \,, \quad d^\dagger (\mu \cdot \eta) = (d^\dagger \mu) \cdot \eta + \mu (d^\dagger \eta) \,, \]
where the multiplicative factor $i$ arises from the definition of the pairing $\langle\blank,\blank\rangle_!$.
This implies that $E^\dagger_\rho = \exp(\rho d^\dagger)$ with the explicit formula given by
\[ E^\dagger_\rho \mu = \sum_{p \leq \mu} {\mu \choose p} \rho^{\sum_{i=1}^k p_i} (\mu \minus p) \,, \quad \text{for } \mu \in \N^k \,, \]
where $p \leq \mu$ if and only if $p_i \leq \mu_i$ for $i=1,\dots,k$ and $p \in \N^k$, $\binom{\mu}{p} := \frac{\mu!}{(\mu \setminus p)! p!} = \prod_{i=1}^k \binom{\mu_i}{p_i}$ denote the multi-index binomial coefficient, and $\mu \minus p$ denotes the multi-index given by $(\mu_1 - p_1, \dots, \mu_k - p_k)$.
Moreover, note that for $\mu, \eta, \nu \in \N^\infty$, we have,
\[ \langle \mu \cdot \eta, \nu \rangle_! = \langle \mu \otimes \eta, \Delta (\nu) \rangle_! \,, \quad \text{with } \Delta (\nu) = \sum_{\nu_1 \cdot \nu_2 = \nu} \nu_1 \otimes \nu_2 \,, \]
where we use the property $\nu! = \nu_1! \cdot \nu_2!$.

\begin{proposition}
    \label{prop:graft_adjoint}
    Consider $\alpha_1, \alpha_2 : \N^\infty \to \R$ and $\rho \in \R$, then,
    \[ \delta(\alpha_1) \cdot E_\rho \delta(\alpha_2) = \delta\big( \alpha_1 *_\rho \alpha_2 \big) \,, \]
    where the product $*_\rho$ is defined by
    \[ (\alpha_1 *_\rho \alpha_2) (\mu) := \sum_{\mu_1 \cdot \mu_2 = \mu} \alpha_1 (\mu_1) \alpha_2 (E^\dagger_\rho \mu_2) \,. \]
\end{proposition}
\begin{proof}
    It is clear that the product $\delta(\alpha_1) \cdot E_\rho \delta(\alpha_2)$ is equal to $\delta\big( \alpha_1 *_\rho \alpha_2 \big)$ for some coefficient map $\alpha_1 *_\rho \alpha_2$. To compute the coefficient map, we take the adjoints of concatenation and $E_\rho$ as follows,
    \begin{align*}
        (\alpha_1 *_\rho \alpha_2)(\mu) &= \langle \delta\big( \alpha_1 *_\rho \alpha_2 \big), \mu \rangle_! = \langle \delta(\alpha_1) \cdot E_\rho \delta(\alpha_2), \mu \rangle_! \\
                                        &= \sum_{\mu_1 \cdot \mu_2 = \mu} \langle \delta(\alpha_1), \mu_1 \rangle_! \langle E_\rho \delta(\alpha_2), \mu_2 \rangle_! = \sum_{\mu_1 \cdot \mu_2 = \mu} \alpha_1 (\mu_1) \alpha_2 (E^\dagger_\rho \mu_2) \,,
    \end{align*}
    where we use the property $\langle a \otimes b, c \otimes d \rangle_! = \langle a, c \rangle_! \langle b, d \rangle_!$.
\end{proof}

Consider the semidirect product $*$ on $\R \times (\N^\infty \to \R)$ given by
\[ (\rho_1, \alpha_1) * (\rho_2, \alpha_2) = (\rho_1 + \rho_2, \alpha_1 *_{\rho_1} \alpha_2) \,, \]
for $\rho_1, \rho_2 \in \R$ and $\alpha_1, \alpha_2 : \N^\infty \to \R$. Note that $*$ is associative.
We are now ready to write the multi-index series of a composition of two splitting methods in terms of their coefficient maps.

\begin{theorem}[Composition law]
    \label{thm:composition}
    Let $\Phi_h$ and $\Psi_h$ be two splitting methods expanded in multi-index series with coefficient maps $\alpha_1$ and $\alpha_2$, respectively. Then their composition $\Psi_{\rho_2 h} \circ \Phi_{\rho_1 h}$ can be expanded in multi-index series as
    \[ \Psi_{\rho_2 h} \circ \Phi_{\rho_1 h} = B\Big( (\rho_1, \alpha_{1,\rho_1}) * (\rho_2, \alpha_{2,\rho_2}) \Big) = e^{(\rho_1 + \rho_2) hA} B(\alpha_{1,\rho_1} *_{\rho_1} \alpha_{2,\rho_2}) \,, \]
    with the explicit formula for the coefficient map given by
    \[ (\alpha_{1,\rho_1} *_{\rho_1} \alpha_{2,\rho_2}) (\mu) = \sum_{\substack{\mu_1 \cdot \mu_2 = \mu \\ p \leq \mu_2}} {\mu_2 \choose p} \rho_1^{|\mu_1| + s(p)} \rho_2^{|\mu_2| - s(p)} \alpha_1 (\mu_1) \alpha_2 (\mu_2 \setminus p) \,, \]
    for any $\mu \in \N^\infty$ with $s(p) := \sum_{i=1}^k p_i$ for $\mu_2 \in \N^k$.
\end{theorem}
\begin{proof}
    The statement follows from the following computation where we use extended multi-index series represented using the map $\delta$,
    \begin{align*}
        \Psi_{\rho_2 h} \circ \Phi_{\rho_1 h} &= \dF\big(\rho_2, \delta(\alpha_{2,\rho_2})) \dF(\rho_1, \delta(\alpha_{1,\rho_1})\big) \\
        \intertext{apply Proposition \ref{prop:graft} to the product of $\dF$}
                                              &= \dF \Big( (\rho_1, \delta(\alpha_{1,\rho_1})) \graft (\rho_2, \delta(\alpha_{2,\rho_2})) \Big) \\
                                              \intertext{use the definition of the semidirect product $\graft$}
                                              &= \dF \big( \rho_1 + \rho_2, \delta(\alpha_{1,\rho_1}) \cdot E_{\rho_1} \delta(\alpha_{2,\rho_2}) \big) \\
                                              \intertext{apply Proposition \ref{prop:graft_adjoint} to the product of $\delta$}
                                              &= \dF \big( \rho_1 + \rho_2, \delta(\alpha_{1,\rho_1} *_{\rho_1} \alpha_{2,\rho_2}) \big) \\
                                              \intertext{recover multi-index series from $\delta$}
                                              &= B(\rho_1 + \rho_2, \alpha_{1,\rho_1} *_{\rho_1} \alpha_{2,\rho_2}) = B\Big( (\rho_1, \alpha_{1,\rho_1}) * (\rho_2, \alpha_{2,\rho_2}) \Big) \,,
    \end{align*}
    with the last equality using the definition of the semidirect product $*$. The explicit formula for the coefficient map follows from the definition of $*_{\rho_1}$ and $E^\dagger_{\rho_1}$.
\end{proof}

\begin{remark}[Comparison with extended word series]\label{rem:extended-word-comparison}
Extended word series are considered~\cite[Section 4]{Sanz-SernaFSN15} in the perturbed integrable setting with the composition law expressed by a semidirect product. However, the roles of the additional coordinates and the corresponding actions are different.
\end{remark}

The composition law can be used to analyze both the processing of splitting methods \cite{BlanesHRK01} and the construction of higher-order splitting methods from lower-order ones through composition.

\begin{remark}[Composition methods]
\label{rem:composition_methods}
Let $\Phi_h=e^{hA}B(\alpha)$ be of order $p\geq1$. Consider the following composition method with $s \in \N$ and $\rho_1,\dots,\rho_s \in \R$ such that $\sum_{i=1}^s \rho_i = 1$,
\[ \Psi_h = \Phi_{\rho_s h} \circ \cdots \circ \Phi_{\rho_1 h} = B\Big((\rho_1, \alpha_{\rho_1}) * \cdots * (\rho_s, \alpha_{\rho_s})\Big) \,. \]
Let us use the notation $(\rho, \alpha)(\mu) = \alpha(\mu)$ for simplicity. We use the fact that $\alpha(\mu) = \frac1{\gamma(\mu)}$ for $\mu \in \N^\infty, |\mu| \leq p,$ to obtain for $\nu \in \N^\infty$ with $|\nu| = p + 1$,
\begin{align*}
    \big((\rho_1, \alpha_{\rho_1}) * \cdots &* (\rho_s, \alpha_{\rho_s})\big) (\nu) - \frac1{\gamma}(\nu) \\
                                            &= \big((\rho_1, \alpha_{\rho_1}) * \cdots * (\rho_s, \alpha_{\rho_s})\big) (\nu) - \big((\rho_1, \frac1{\gamma}_{\rho_1}) * \cdots * (\rho_s, \frac1{\gamma}_{\rho_s})\big)(\nu) \\
                                                                                         &= \big( \sum_{i=1}^s\rho_i^{p+1} \big) \big( \alpha(\nu) - \frac1{\gamma}(\nu) \big) \,,
\end{align*}
since all terms involving multi-indices of size less than $p+1$ cancel. 
Imposing $\sum_{i=1}^s\rho_i^{p+1}=0$ raises the order to at least $p+1$ which recovers the classical result.
\end{remark}

\begin{remark}[Comparison with multi-index B-series]\label{rem:commutative-reduction}
Let $\mathcal{P}=\R[z_0,z_1,\ldots]$ be the polynomial algebra in commuting variables, and let $D$ be its derivation defined by $Dz_i=z_{i+1}$. On the finite concatenation algebra $\MM$, define
\[
 \varphi(\one)=1,\qquad
 \varphi((\mu_1,\ldots,\mu_k))=z_{\mu_1}\cdots z_{\mu_k}.
\]
Then $\varphi(xy)=\varphi(x)\varphi(y)$ and $\varphi(dx)=D\varphi(x)$. Consequently, the infinitesimal operation $x\mathbin{\to}y:=x\cdot dy = x \cdot \frac{d}{d\rho} (E_\rho y) |_{\rho=0}$ satisfies
\[
 \varphi(x\mathbin{\to}y)
 =\varphi(x)D\varphi(y).
\]
The same Novikov product is used in~\cite[Section 2]{BrunedMB25} to describe the composition law. Note that $\varphi$ forgets the ordering of letters. Moreover, the monomials $\prod_i z_i^{r_i}$ indexing rooted trees in \cite{BrunedMB25} satisfy the population condition $\sum_{i\geq0}(1-i)r_i=1$, whereas the images of arbitrary words here need not satisfy it.
\end{remark}

\section{Stage-dependent substitution}
\label{sec:substitution}

We extend the notation of Section~\ref{sec:intro_multi_index} by attaching stage labels to multi-indices.

\subsection{Decorated multi-index series}
\label{sec:butcher_splitting}

Convergence order analysis of generalized modified splitting methods \eqref{eq:splitting_gen} leads to the introduction of \emph{decorated multi-indices}. Let the set $\Lambda_k$ be defined as
\[ \Lambda_k := \{ \lambda \in \N^k \; | \; 1 \leq \lambda_1 \leq \cdots \leq \lambda_k \leq s \} \,. \]

\begin{definition}
    A \emph{decorated multi-index} is a pair $(\mu, \lambda)$ of a multi-index $\mu \in \N^k$ for some $k \in \N$ and a decoration $\lambda \in \Lambda_k$. The set of decorated multi-indices is denoted by $\N^\infty_\Lambda$.
\end{definition}

The correspondence map $\dF$ between decorated multi-indices and differential operators is defined as
\[ \dF(\mu, \lambda) = h^{|\mu|} \ad_{-A}^{\mu_k} (B_{\lambda_k}) \cdots \ad_{-A}^{\mu_1} (B_{\lambda_1}) \,, \]
where $\mu \in \N^k$ and $\lambda \in \Lambda_k$ and $\dF(\one, \one) = \id$.

Let the symmetric sum over the set $\Lambda_k$ be defined as
\[ \sum^!_{\lambda \in \Lambda_k} := \sum_{\lambda \in \Lambda_k} \frac1{q_1! \cdots q_s!} \,, \]
where $q_i$ is the number of occurrences of $i$ in $\lambda$.

\begin{definition}
    \label{def:decorated_butcher}
    \emph{Decorated multi-index series} are defined as
    \[ B_\Lambda(\alpha) := \sum^!_{(\mu, \lambda) \in \N^\infty_\Lambda} \frac{\alpha(\mu, \lambda)}{\mu!} \dF(\mu, \lambda) \,, \]
    where $\alpha : \N^\infty_\Lambda \to \R$ is a coefficient map.
\end{definition}

The proof of Proposition~\ref{prop:splitting_expansion} relies on the formal identity
\begin{equation}
    \label{eq:commute_exp}
    e^A e^B e^{-A} = e^{\exp(\ad_A)(B)} \,,
\end{equation}
where $\exp(\ad_A) = \id + \sum_{k=1}^\infty \frac{1}{k!} \ad^k_A $ is the operator exponential.

\begin{proposition}
    \label{prop:splitting_expansion}
    Consider a generalized modified splitting method \eqref{eq:splitting_gen} with coefficients $a_j, b_j \in \R$, $j=1,\dots,s$, satisfying $\sum_{j=1}^s a_j=1$. Then it has the decorated multi-index expansion
    \[ u_{n+1} = e^{hA} B_\Lambda(\alpha) u_n \,, \]
    for $\alpha : \N^\infty_\Lambda \to \R$ defined for $\mu \in \N^k, \lambda \in \Lambda_k$ and $c_{\lambda_i} = \sum_{j=1}^{\lambda_i} a_j$ by
    \[ \alpha (\mu, \lambda) = \prod_{i=1}^k b_{\lambda_i} c_{\lambda_i}^{\mu_i} \,. \]
\end{proposition}
\begin{proof}
    Given the splitting method of the form \eqref{eq:splitting_gen}, we apply \eqref{eq:commute_exp} repeatedly to move all exponentials $e^{a_j hA}$ to the left, obtaining
    \begin{equation}
        \label{eq:exp}
        u_{n+1} = e^{h A} \prod_{j=1}^s e^{b_{s-j+1} h \exp(\ad_{-c_{s-j+1} h A}) (B_{s-j+1})} u_n \,,
    \end{equation}
    where we use the assumption that $c_s = 1$. We obtain
    \begin{align*}
        &\prod_{j=1}^s e^{b_{s-j+1} h \exp(\ad_{-c_{s-j+1} h A}) (B_{s-j+1})} \\
        &\quad = \sum_{r \in \N^s} h^{\sum_{j=1}^s r_j}
            \prod_{j=1}^s \frac{b_{s-j+1}^{r_{s-j+1}}}{r_{s-j+1}!}
            \exp(\ad_{-c_{s-j+1} h A}) (B_{s-j+1})^{r_{s-j+1}} \\
        &\quad = \sum_{k=0}^\infty \sum_{\lambda \in \Lambda_k}^! h^k
            \prod_{l=1}^k b_{\lambda_{k-l+1}}
            \exp(\ad_{-c_{\lambda_{k-l+1}} h A}) (B_{\lambda_{k-l+1}})
        \intertext{Expanding $\exp(\ad_x) = \sum_{i=0}^\infty \frac{1}{i!} \ad_x^i$, we obtain}
        &\quad = \sum_{k=0}^\infty \sum_{\mu \in \N^k} \frac{h^{|\mu|}}{\mu!}
            \sum_{\lambda \in \Lambda_k}^!
            \prod_{l=1}^k b_{\lambda_{k-l+1}} c_{\lambda_{k-l+1}}^{\mu_{k-l+1}}
            \ad_{-A}^{\mu_{k-l+1}} (B_{\lambda_{k-l+1}})
        \intertext{By the definition of decorated multi-index series, this equals}
        &\quad = \sum^!_{(\mu, \lambda) \in \N^\infty_\Lambda}
            \frac{\alpha(\mu, \lambda)}{\mu!} \dF(\mu,\lambda)
            = B_\Lambda (\alpha) \,. 
    \end{align*}
    We thus obtain $u_{n+1} = e^{hA} B_\Lambda(\alpha) u_n$ with $\alpha(\mu, \lambda) = \prod_{i=1}^k b_{\lambda_i} c_{\lambda_i}^{\mu_i}$ for $\mu \in \N^k$ and $\lambda \in \Lambda_k$.
\end{proof}

Note that for the case $B_1 = \dots = B_s = B$, we recover the expansions of the splitting methods \eqref{eq:splitting} obtained in \cite{BlanesNFS13,ThalhammerCAH12,ThalhammerHEO08}.

\begin{corollary}
    \label{cor:splitting_expansion}
    If $B_1 = \dots = B_s = B$, then
    \[ u_{n+1} = e^{hA} B(\alpha) u_n \,, \]
    with the coefficient $\alpha(\mu) = \sum_{\lambda \in \Lambda_k}^! \prod_{i=1}^k b_{\lambda_i} c_{\lambda_i}^{\mu_i}$ for $\mu \in \N^k$.
\end{corollary}

\begin{remark}[Relation to frozen-flow order conditions]\label{rem:frozen-flow-comparison}
The order conditions for splitting methods are in bijection with a subfamily of order conditions for frozen-flow methods~\cite[Theorem 4.2 and Corollary 4.3]{OwrenRMA99}. Let $C_m$ denote a rooted tree whose root has $m$ leaf children, with $C_0$ a single vertex. Associate to $\mu=(\mu_1,\ldots,\mu_k)$ the planar tree
\[
 T_\mu=[C_{\mu_1},\ldots,C_{\mu_k}],
\]
where the brackets attach the listed trees, in their given order, to a new root. The empty word corresponds to the single-vertex tree. This is a bijection onto the planar rooted trees of height at most two, and $T_\mu$ has $|\mu|+1$ vertices.

The similar coefficient formulas in this paper and in \cite{OwrenRMA99} reflect the common expansion of ordered products of exponentials, see Lemma 4.1, Theorem 4.2, and Corollary 4.3 of~\cite{OwrenRMA99}.
\end{remark}

\subsection{Substitution of decorated multi-index series}

Let $\MM_+$ be the span of the nonempty multi-indices and let $T(\MM_+)$ be its tensor algebra, with product $\odot$ and unit $\one_T$. 
The \emph{insertion} product $\ins : T(\MM_+) \otimes \MM \to \MM$ is defined for $\mu \in \N^k$ and nonempty $\eta_1,\dots,\eta_k \in \N^\infty$ by
\[ \eta_1 \odot \cdots \odot \eta_k \ins \mu := (d^{\mu_1} \eta_1) \cdots (d^{\mu_k} \eta_k) \,, \]
where $d : \MM \to \MM$ is the operator defined in Section~\ref{sec:composition}.
For $l \neq k$, we have $\eta_1 \odot \cdots \odot \eta_l \ins \mu = 0$. The unit convention is $\one_T\ins\one=\one$.

\begin{proposition}
    \label{prop:insertion}
    Let $\mu \in \N^k$ and let $\eta_1,\dots,\eta_k$ be nonempty multi-indices. Then
    \[ h^{\sum_{i=1}^k\mu_i}\ad^{\mu_k}_{-A} (\dF(\eta_k)) \cdots \ad^{\mu_1}_{-A} (\dF(\eta_1)) = \dF(\eta_1 \odot \cdots \odot \eta_k \ins \mu) \,. \]
\end{proposition}
\begin{proof}
    The statement follows from the definition of $\dF$ and the fact that $h^k \ad^k_{-A} (\dF(\eta)) = \dF(d^k \eta)$ for any nonempty multi-index $\eta$.
\end{proof}

Define $\Delta_\ins : \MM \to T(\MM_+) \otimes \MM$ by $\Delta_\ins(\one)=\one_T\otimes\one$ and, for nonempty $\mu$, by
\[ \Delta_\ins (\mu) := \sum_{\substack{\pi \in P(\mu) \\ p \in \N^k}} \frac1{p!} (d^{\dagger p_1}\pi_1) \odot \cdots \odot (d^{\dagger p_k} \pi_k) \otimes p \,, \]
where $P(\mu)$ is the set of all partitions of $\mu$, that is,
\[P(\mu) := \{ (\pi_1, \dots, \pi_k)  \; | \; \pi_1 \cdots \pi_k = \mu \,,  \one \neq \pi_i \in \N^\infty \} \,, \]
where $k \in \N$ is the number of parts in $\pi$. Note the following explicit formula,
\begin{equation}
    \label{eq:dagger_explicit}
    \frac{d^{\dagger k} \mu}{k!} = \sum_{\substack{p \leq \mu \\ \sum_i p_i = k}} \binom{\mu}{p} (\mu \setminus p) \,,
\end{equation}
which is used to obtain an explicit formula for the substitution law in Theorem \ref{thm:substitution}.

\begin{proposition}
    \label{prop:insertion_adjoint}
    The coproduct $\Delta_\ins : \MM \to T(\MM_+) \otimes \MM$ defined above is adjoint to the insertion product $\ins : T(\MM_+) \otimes \MM \to \MM$ with respect to the pairing $\langle \cdot, \cdot \rangle_!$, that is,
    \[ \langle \eta_1 \odot \cdots \odot \eta_k \ins \mu, \nu \rangle_! = \langle \eta_1 \odot \cdots \odot \eta_k \otimes \mu, \Delta_\ins (\nu) \rangle_! \,, \]
    for nonempty $\eta_1,\dots,\eta_k \in \N^\infty$, $\mu\in\N^k$, and $\nu\in\N^\infty$, including the tensor-unit case $k=0$.
\end{proposition}
\begin{proof}
    The unit case follows from the definitions. For nonempty $\eta_i \in \N^\infty$, $\mu\in\N^k$, and nonempty $\nu \in \N^\infty$, the statement follows from the computation below,
    \begin{align*}
        \langle \eta_1 \odot \cdots \odot \eta_k &\ins \mu, \nu \rangle_! = \langle (d^{\mu_1} \eta_1) \cdots (d^{\mu_k} \eta_k), \nu \rangle_! \\
        \intertext{the adjoint of the concatenation of $k$ non-empty multi-indices is the repeated deconcatenation into $k$ non-empty multi-indices, so we have}
                                                                         &= \langle (d^{\mu_1} \eta_1) \odot \cdots \odot (d^{\mu_k} \eta_k), \sum_{\pi \in P_k(\nu)} \pi_1 \odot \cdots \odot \pi_k \rangle_! \\
                                                                         \intertext{where $P_k(\nu)$ denotes the partitions into $k$ parts, we extended the pairing to tensor products component-wise, so we have}
                                                                         &= \langle \eta_1 \odot \cdots \odot \eta_k, \sum_{\pi \in P_k(\nu)} (d^{\dagger \mu_1} \pi_1) \odot \cdots \odot (d^{\dagger \mu_k} \pi_k) \rangle_! \\
                                                                         \intertext{we remove the terms $\mu_i$ from right-hand side by introducing $p$ and pairing it with $\mu$ as follows,}
                                                                         &= \langle \eta_1 \odot \cdots \odot \eta_k \otimes \mu, \sum_{\substack{\pi \in P_k(\nu) \\ p \in \N^k}} \frac1{p!} (d^{\dagger p_1} \pi_1) \odot \cdots \odot (d^{\dagger p_k} \pi_k) \otimes p \rangle_! \,,
    \end{align*}
    which agrees with the definition of $\Delta_\ins : \MM \to T(\MM_+) \otimes \MM$ after we replace the set $P_k(\pi)$ by $P(\pi)$.
\end{proof}

Let $\beta_\lambda (\eta)$ for $\lambda \in \Lambda_k$ and $\eta = \eta_1 \odot \cdots \odot \eta_k \in T(\MM_+)$ be defined as
\[ \beta_\lambda (\eta) := \beta_{\lambda_1} (\eta_1) \cdots \beta_{\lambda_k} (\eta_k) \,. \]
For $k=0$, the empty decoration is denoted by $\one$, the empty product is $\beta_\one(\one_T)=1$, and $\alpha(\one,\one)=1$ for a splitting method.
\begin{definition}
    \label{def:infinitesimal_character}
    A coefficient map $\beta : \N^\infty \to \R$ is called an
    \emph{infinitesimal character} if $\beta(\one)=0$ and
$\beta(\mu\shuffle\eta)=0$ for all nonempty multi-indices $\mu,\eta$.
\end{definition}

A key property of infinitesimal characters $\beta$ satisfying $\beta(\mu) \neq 0$ for finitely many $\mu \in \N^\infty$ is that they correspond to elements of $\LL$ given by $B(\beta)$ \cite{ReeLEA58}.
We can now express the decorated expansion derived in Proposition \ref{prop:splitting_expansion} as an ordinary multi-index series.

\begin{theorem}[Substitution law]
    \label{thm:substitution}
    Let $\Phi_h = e^{hA}B_\Lambda(\alpha)$ be a generalized modified splitting method \eqref{eq:splitting_gen} with $\sum_j a_j=1$. Suppose $B_j(h)=h^{-1}B(\beta_j) \in \LL$. Then
    \[ B_\Lambda(\alpha)=B(\beta\star\alpha),\qquad \Phi_h=e^{hA}B(\beta\star\alpha), \]
    with the coefficient map $\beta \star \alpha : \N^\infty \to \R$ given by $(\beta \star \alpha) (\one) = \alpha(\one, \one) = 1$ and
    \[ (\beta \star \alpha) (\mu) := \sum_{\substack{\pi \in P(\mu) \\ p_i \leq \pi_i}} \sum_{\lambda \in \Lambda_k}^! \Big( \prod_{i=1}^k \binom{\pi_i}{p_i} \beta_{\lambda_i} (\pi_i \minus p_i) \Big) \alpha(p_\Sigma, \lambda) \,, \]
    where $k$ is the number of parts in $\pi$ and $p_\Sigma := (\sum_j p_{1j}, \dots, \sum_j p_{kj})$.
\end{theorem}
\begin{proof}
    Note that $h^{-1}B(\beta_j) \in \LL$ implies that $\beta_j$ is an infinitesimal character.
    Consider decorated multi-index series $B_\Lambda (\alpha)$ and $B(\beta_j)$ for $j=1,\dots,s$. Substituting $B_j=h^{-1}B(\beta_j)$ into a decorated elementary operator gives, for $\mu\in\N^k$,
    \[ \dF(\mu,\lambda)=h^{|\mu|-k}
    \ad_{-A}^{\mu_k}\big(B(\beta_{\lambda_k})\big)\cdots
    \ad_{-A}^{\mu_1}\big(B(\beta_{\lambda_1})\big). \]
    Since $|\mu|-k=\sum_i\mu_i$, apply Proposition~\ref{prop:insertion} using $B(\beta_j)=\dF(\delta(\beta_j))$. Thus
    \begin{align*}
        B_\Lambda (\alpha) &= \sum_{k=0}^\infty \sum_{\mu\in\N^k}\sum^!_{\lambda\in\Lambda_k} \frac{\alpha(\mu, \lambda)}{\mu!} \dF(\mu, \lambda) \\
                           &= \sum_{k=0}^\infty \sum_{\lambda \in \Lambda_k}^! \dF\Big( \delta(\beta_{\lambda_1}) \odot \cdots \odot \delta(\beta_{\lambda_k}) \ins \delta_k \big(\alpha(\blank, \lambda)\big) \Big) \,,
    \end{align*}
    where $\alpha(\blank,\lambda)$ is regarded as a map on $\N^k$ and $\delta_k(\alpha)$ the part of $\delta(\alpha)$ indexed by multi-indices of length $k$, that is,
\[ \delta_k(\alpha):=\sum_{\mu\in\N^k}\frac{\alpha(\mu)}{\mu!}\mu \quad \text{with } \delta(\alpha) = \sum_{k=0}^\infty \delta_k (\alpha) \,. \]
    The $k=0$ term is the identity.
    This implies that $B_\Lambda (\alpha)$ can be expressed as a multi-index series $B(\beta \star \alpha)$ for some coefficient map $\beta \star \alpha : \N^\infty \to \R$ which we compute as follows:
    \begin{align*}
        (\beta \star \alpha) (\mu) &= \sum_{k=0}^\infty \sum_{\lambda \in \Lambda_k}^! \langle \delta(\beta_{\lambda_1}) \odot \cdots \odot \delta(\beta_{\lambda_k}) \ins \delta_k \big(\alpha(\blank, \lambda)\big), \mu \rangle_! \\
                                   &= \sum_{k=0}^\infty \sum_{\lambda \in \Lambda_k}^! \langle \delta(\beta_{\lambda_1}) \odot \cdots \odot \delta(\beta_{\lambda_k}) \otimes \delta_k \big(\alpha(\blank, \lambda)\big), \Delta_\ins (\mu) \rangle_! \\
                                   &= \sum_{\substack{\pi \in P(\mu) \\ p \in \N^k}} \frac1{p!} \sum_{\lambda \in \Lambda_k}^! \Big( \prod_{i=1}^k \beta_{\lambda_i} (d^{\dagger p_i}\pi_i) \Big) \alpha(p, \lambda) \,,
    \end{align*}
    where we use the natural extension of the pairing $\langle \blank, \blank \rangle_!$ to $T(\MM_+) \otimes \MM$. The explicit formula for the coefficient map follows from \eqref{eq:dagger_explicit}.
\end{proof}

\begin{remark}[Stage-dependent substitution and modified vector fields]\label{rem:substitution-comparison}
For classical B-series, substitution replaces the vector field defining the elementary differentials by another series~\cite{ChartierASB10a}. The substitution law above serves a corresponding purpose for generalized modified splitting methods with the additional flexibility of replacing the operator $B$ at stage $j$ by $h^{-1}B(\beta_j)$. This stage-dependent bookkeeping permits a correction to be localized to a chosen stage, as in the single-stage order-raising theorem below.
\end{remark}

See Appendix~\ref{app:order_conditions} for an application of the substitution law to the derivation of the order-six conditions for symmetric generalized modified splitting methods. In practice, some nested commutators of $A$ and $B$ may be expensive to compute. It may therefore be desirable to set $\beta_j$ to zero for such commutators. This reduces the number of degrees of freedom available for satisfying the order conditions and may consequently require an increase in the number of stages $s$ of the generalized modified splitting method.

The substitution law can also be applied to the formal order analysis of Zassenhaus splitting methods, as briefly discussed in Section~\ref{sec:literature}. This application is outside the scope of the present paper and is left for future work.

Note that the coefficient map $\beta \star \alpha$ satisfies,
\[ (\beta \star \alpha) (\mu \shuffle \eta) = (\beta \star \alpha) (\mu) (\beta \star \alpha) (\eta) \,, \]
that is, $\beta \star \alpha$ is a character with respect to the shuffle product. This follows from the fact that a generalized modified splitting method $\Phi_h$ can be written as $\Phi_h = e^{hA} e^C$ for $C$ a formal sum of elements in $\LL$ given by \eqref{eq:exp} and the repeated application of the BCH formula.

The obstruction results in \cite{BeauchardCTS26} give useful guidance when choosing commutators. Let $W_r=[\ad_{-A}^{r-1} B, \ad_{-A}^r B]$, Theorem~1.18 states that, for $n\geq1$, methods constructed using $A$ and $B,W_1,\dots,W_{n-1} \in \LL$ attain order $2n$ with nonnegative steps along $A$, and that $2n$ is the maximal order in this class. The bracket $W_r$ is associated with the Lyndon multi-index $(r-1,r)$ of size $2r+1$. For $n\geq2$, the highest bracket in this family therefore has degree $2n-1$. The possibility of other choices for $W_r$ is discussed as an open question in \cite{BeauchardCTS26}.

\section{High-order generalized modified splitting methods}\label{sec:high_order_construction}

We use the substitution law to develop a systematic procedure for constructing high-order generalized modified splitting methods from lower-order ones. The procedure consists of identifying the leading error term of a given generalized modified splitting method and incorporating it into any chosen operator $B_j$ in \eqref{eq:splitting_gen}. The leading error term of the resulting integrator is canceled, thereby increasing its order.

\begin{theorem}[Single-stage order raising]
    \label{thm:modified_eq}
    Consider a generalized modified splitting method of order $p\geq1$ satisfying the hypotheses of Theorem~\ref{thm:substitution}, with coefficients $a_j,b_j$, $j=1,\dots,s$, and $B_j=h^{-1}B(\beta_j)$. Choose $\hat{j} \in \{1, \dots, s\}$ such that $b_{\hat{j}}\neq 0$ and $\beta_{\hat{j}} (\mu) = 0$ for all $|\mu| \geq p+1$, and define
    \begin{align*}
        \tilde{\beta}_{\hat{j}} (\mu) &:= \beta_{\hat{j}} (\mu) \,, & &|\mu| < p + 1 \,, \\
        \tilde{\beta}_{\hat{j}} (\mu) &:= \frac{1}{\alpha((0), (\hat{j}))} \Big( \frac{1}{\gamma(\mu)} - (\beta \star \alpha) (\mu) \Big) \,, & &|\mu| = p+1 \,.
    \end{align*}
    Then replacing $\beta_{\hat j}$ by $\tilde\beta_{\hat j}$ gives a generalized modified splitting method of order $p+1$.
\end{theorem}
\begin{proof}
    We start by checking that $\tilde\beta_{\hat{j}}$ is an infinitesimal character. Consider nonempty $\mu, \eta \in \N^\infty$ such that $|\mu| + |\eta| < p+1$, then the infinitesimal character property of $\beta_{\hat{j}}$ implies $\tilde\beta_{\hat{j}} (\mu \shuffle \eta) = 0$. For $|\mu| + |\eta| = p+1$, we use the fact that the method is of order $p$ with $|\mu|, |\eta| \leq p$ which implies $\tilde\beta_{\hat{j}} (\mu \shuffle \eta) = 0$ due to the character property of $\beta \star \alpha$ and $\frac1{\gamma}$.

    The new method is of order $p$ due to $\tilde{\beta}_{\hat{j}} (\mu) = \beta_{\hat{j}} (\mu)$ for $|\mu| \leq p$. Let $\tilde\beta = (\beta_1, \dots, \tilde{\beta}_{\hat{j}}, \dots, \beta_s)$. It remains to check that $(\tilde\beta \star \alpha) (\mu) = 1/\gamma(\mu)$ for $|\mu| = p+1$. Using the definition of the substitution law and the notation
    \[ \Delta_\ins (\mu) = \sum_{(\mu)} \mu_{(1)} \otimes \mu_{(2)} = \mu \otimes (0) + \sum_{\widetilde{(\mu)}} \mu_{(1)} \otimes \mu_{(2)} \,, \]
    we obtain
    \begin{align*}
        (\tilde\beta \star \alpha) (\mu) &= \sum_{(\mu)} \sum^!_{\lambda \in \Lambda_{\#\mu_{(2)}}} \tilde{\beta}_\lambda (\mu_{(1)}) \alpha(\mu_{(2)}, \lambda) \\
        &= \tilde{\beta}_{\hat{j}} (\mu) \alpha((0), (\hat{j}))
            + \sum_{\substack{j=1\\ j\neq \hat{j}}}^s \beta_j (\mu) \alpha((0), (j)) \\
        &\quad + \sum_{\widetilde{(\mu)}} \sum^!_{\lambda \in \Lambda_k}
            \beta_\lambda (\mu_{(1)}) \alpha(\mu_{(2)}, \lambda) \,,
            \intertext{where $\mu_{(2)} \in \N^k$. Using the assumption $\beta_{\hat{j}} (\mu) = 0$, we obtain}
        &= \tilde{\beta}_{\hat{j}} (\mu) \alpha((0), (\hat{j})) + (\beta \star \alpha) (\mu) = \frac{1}{\gamma(\mu)} \,.
    \end{align*}
    This finishes the proof.
\end{proof}

The process in Theorem~\ref{thm:modified_eq} can be iterated. For a generalized modified splitting method with odd $s$ to be symmetric, the sufficient conditions are
\[ a_1=0,\quad a_j=a_{s-j+2}\ (2\leq j\leq s),\quad b_j=b_{s-j+1},\quad B_j(h)=B_{s-j+1}(h),\quad B_j(-h)=B_j(h) \,, \]
for even $s$, the sufficient conditions are
\[ b_s=0,\quad a_j=a_{s-j+1},\quad b_j=b_{s-j}\ (1\leq j\leq s-1),\quad B_j(h)=B_{s-j}(h),\quad B_j(-h)=B_j(h) \,, \]
If the central stage $\hat j=(s+1)/2$ for odd $s$ and $\hat j=s/2$ for even $s$ satisfies the hypotheses of Theorem~\ref{thm:modified_eq}, its correction to $B_{\hat j}$ is proportional to $h^p$ with even power. Therefore, the resulting method is symmetric.

\subsection{A sixth-order method based on Chin's composition}
\label{sec:high_order_chin}

We apply Theorem~\ref{thm:modified_eq} to the symmetric second-order splitting method
\begin{equation}
    \label{eq:classic_splitting}
    u_{n+1} = e^{\frac16 h B} e^{\frac12 hA} e^{\frac23 hB} e^{\frac12 hA} e^{\frac16 h B} u_n \,.
\end{equation}
The splitting method \eqref{eq:classic_splitting} is taken from the seminal work of Chin \cite{ChinSIC97}.
To apply Theorem~\ref{thm:modified_eq}, we need to write the splitting method \eqref{eq:classic_splitting} as a generalized modified splitting method. We do so by defining the coefficients $a_j, b_j$ for $j=1,\dots,3$ and the coefficient maps $\beta_j$ for operators $B_j = \frac1h B(\beta_j)$ as follows:
\[ a_1 = 0 \,, \quad a_2 = \frac12 \,, \quad a_3 = \frac12 \,, \quad b_1 = \frac16 \,, \quad b_2 = \frac23 \,, \quad b_3 = \frac16 \,, \]
\[ \beta_1((0)) = 1 \,, \quad \beta_2((0)) = 1 \,, \quad \beta_3((0)) = 1 \,, \]
with $\beta_j (\mu) = 0$ for all the other multi-indices $\mu \in \N^\infty$ not listed above.

We apply Theorem~\ref{thm:modified_eq} with $\hat{j} = 2$ to derive a fourth-order generalized modified splitting method that coincides with Chin's splitting method \cite{ChinSIC97}. We note that the coefficients $a_j, b_j$ in \eqref{eq:classic_splitting} are chosen such that the order condition corresponding to the multi-index $(2)$ is satisfied, that is,
\[ (\beta \star \alpha)((2)) = \frac1{\gamma((2))} \,. \]
The degree-three Lyndon coefficients of the corrected map $\tilde{\beta}_2$ are therefore
\[ \tilde{\beta}_2((2)) = 0 \,, \quad \tilde{\beta}_2((0,1)) = \frac1{48} \,. \]
The resulting fourth-order method is
\[ u_{n+1} = e^{\frac16 h B} e^{\frac12 hA} e^{\frac23 h(B + \frac1{48} h^2 [B,[A,B]])} e^{\frac12 hA} e^{\frac16 h B} u_n \,. \]
This is the splitting method with modified potential introduced in \cite{ChinSIC97}.

Applying Theorem~\ref{thm:modified_eq} again with $\hat{j} = 2$ gives a sixth-order method. The degree-five Lyndon coefficients of the corrected map $\tilde{\beta}_2$ are
\begin{align*}
    \tilde\beta_2((4)) &= - \frac1{80} \,, &\tilde\beta_2((0,3)) &= - \frac7{960} \,, \\
    \tilde\beta_2((1,2)) &= - \frac1{960} \,, &\tilde\beta_2((0,0,2)) &= - \frac7{4320} \,, \\
    \tilde\beta_2((0,1,1)) &= \frac1{2160} \,, &\tilde\beta_2((0,0,0,1)) &= - \frac{41}{103680} \,.
\end{align*}
The resulting sixth-order method is
\[ u_{n+1} = e^{\frac16 h B} e^{\frac12 hA} e^{\frac23 h\tilde{B}} e^{\frac12 hA} e^{\frac16 h B} u_n \,, \]
where
\begin{align*}
    \tilde{B} = B &+ \frac1{48} h^2 [B,[A,B]] - \frac1{80 \cdot 24} h^4 [A, [A, [A, [A, B]]]] - \frac7{960 \cdot 6} h^4 [B, [A, [A, [A, B]]]] \\
                  &- \frac1{960 \cdot 2} h^4 [[A,B], [A, [A, B]]] - \frac7{4320 \cdot 2} h^4 [B, [B, [A, [A, B]]]] \\
                  &+ \frac1{2160} h^4 [[B, [A, B]], [A, B]] - \frac{41}{103680} h^4 [B, [B, [B, [A, B]]]] \,.
\end{align*}

Section \ref{SectionConvergence} verifies numerically the sixth-order convergence of the method. The coefficients of the method were computed using the Python scripts discussed in Section \ref{sec:automatic_computation}.

\subsection{A sixth-order method based on Gauss nodes}
\label{sec:high_order_gauss}

In the Chin-based construction of Section~\ref{sec:high_order_chin}, the central modification contains the commutator $\ad_A^4(B)$. We now choose the principal splitting coefficients so that the single-letter order conditions are satisfied through order six. This eliminates the need for corrections involving $\ad_A^2(B)$ and $\ad_A^4(B)$.

For an ordinary splitting method, the order conditions indexed by the single-letter multi-indices $(m)$ take the form
\[ \alpha((m))=\sum_{j=1}^s b_j c_j^m=\frac1{m+1},
   \qquad c_j=\sum_{i=1}^j a_i. \]
Since $|(m)|=m+1$, the conditions through order six correspond to $m=0,\dots,5$. These are precisely the conditions for a quadrature rule on $[0,1]$ to integrate polynomials of degree at most five exactly:
\[ \sum_{j=1}^s b_j c_j^m=\int_0^1 t^m\,dt,
   \qquad m=0,\dots,5. \]
We therefore choose the three-point Gauss--Legendre nodes and weights,
\[ (c_1,c_2,c_3)=\left(\frac12-q,\frac12,\frac12+q\right),
   \qquad (b_1,b_2,b_3)=\left(\frac5{18},\frac49,\frac5{18}\right),
   \qquad q=\frac{\sqrt{15}}{10}. \]
Setting $r=\frac12-q$, we recover the $A$-coefficients from the cumulative sums $c_j$:
\[ a_1=c_1=r,\qquad a_2=c_2-c_1=q,\qquad
   a_3=c_3-c_2=q,\qquad a_4=1-c_3=r. \]
In the convention of \eqref{eq:splitting_gen}, we set $b_4=0$, so the fourth $B$-stage is the identity. All principal coefficients are nonnegative. Using the unmodified operator $B$ in every $B$-stage gives a symmetric splitting method.

We retain these coefficients and modify only the central $B$-stage:
\begin{equation}
    \label{eq:gauss_sixth}
    \Phi_h^{\mathrm{G}}=
    e^{rhA}e^{\frac5{18}hB}e^{qhA}
    e^{\frac49h\widetilde B_{\mathrm{G}}}
    e^{qhA}e^{\frac5{18}hB}e^{rhA}.
\end{equation}
The operator $B_2=\widetilde B_{\mathrm{G}}$ is determined by the order-raising procedure. The starting method, obtained by setting $\widetilde B_{\mathrm{G}}=B$ in \eqref{eq:gauss_sixth}, is consistent and symmetric, hence of order at least two. We apply Theorem~\ref{thm:modified_eq} twice at $\hat j=2$, canceling first the degree-three error and then the degree-five error. Symmetry gives orders four and six, respectively.

The coefficient indexed by $(4)$ can be affected only by corrections of strictly smaller size and length at most one. Since no such corrections are present, the corresponding order condition remains satisfied, and no correction indexed by $(4)$ is needed. The resulting central operator is
\begin{align}
    \label{eq:gauss_modified_operator}
    \widetilde B_{\mathrm{G}}=B
    &+\theta_1 h^2[B,[A,B]] \notag\\
    &+h^4\Big(\frac{\theta_2}{6}[B,[A,[A,[A,B]]]]
      +\frac{\theta_3}{2}[[A,B],[A,[A,B]]] \notag\\
    &\qquad\quad+\frac{\theta_4}{2}[B,[B,[A,[A,B]]]]
      +\theta_5[[B,[A,B]],[A,B]] \notag\\
    &\qquad\quad+\theta_6[B,[B,[B,[A,B]]]]\Big),
\end{align}
where
\begin{align*}
    \theta_1&=\frac3{16}-\frac{13\sqrt{15}}{288},
    &\theta_2&=\frac9{320}-\frac{13\sqrt{15}}{1920},\\
    \theta_3&=-\frac3{320}+\frac{\sqrt{15}}{384},
    &\theta_4&=\frac1{810},\\
    \theta_5&=\frac{259}{12960}-\frac{\sqrt{15}}{192},
    &\theta_6&=\frac{163}{51840}-\frac{845\sqrt{15}}{1119744}.
\end{align*}
Equivalently, the nonzero Lyndon coefficients of the central modification, apart from $\beta_2((0))=1$, are
\begin{align*}
    \beta_2((0,1))&=\theta_1,&\beta_2((0,3))&=\theta_2,&
    \beta_2((1,2))&=\theta_3,\\
    \beta_2((0,0,2))&=\theta_4,&\beta_2((0,1,1))&=\theta_5,&
    \beta_2((0,0,0,1))&=\theta_6.
\end{align*}
Compared with the Chin-based construction, this method removes the degree-five correction $\ad_A^4(B)$ at the cost of two additional $A$-substeps.

\subsection{Automatic computation}
\label{sec:automatic_computation}

The calculations in Appendix~\ref{app:order_conditions} and the sixth-order constructions above are automated by Python scripts implementing the substitution law, order verification, and order raising. The scripts are available at
\begin{center}
\texttt{https://doi.org/10.5281/zenodo.21976559}, \cite{python_scripts}.
\end{center}

The scripts use only the standard Python library and are documented for adaptation to the user's needs. The scripts include a list of commands used to compute the coefficients corresponding to the Chin-based sixth-order generalized modified splitting method.

\section{Convergence result and numerical experiments}
\label{SectionConvergence}

In this section, we state our main analytic result, a convergence estimate for generalized modified splitting methods applied to linear evolution equations of parabolic or Schr{\"o}dinger type, respectively. 
Moreover, we provide confirming and complementing numerical illustrations. 
We point out that the considered test equation, which takes the simple form of a heat-type equation, arises in ground and excited state computations for linear Schr{\"o}dinger equations by the imaginary time propagation.  
In the context of regular confining potentials, it is reasonable to assume localized solutions. 

\subsection{Convergence result}\label{sec:convergence_result}

For simplicity, we assume that the underlying space domain $\Omega \subset \mathbb{R}^d$ is given as a Cartesian product of bounded intervals. 
With regard to the considered test equations, we focus on (spatial semi-discretizations of) parabolic partial differential equations 
\begin{equation}
\label{eq:pde}
\partial_t u = \Delta u + V(x) \, u\,, \quad (x,t) \in \Omega \times (0, T)\,, 
\end{equation}
involving the Laplacian and a regular real-valued space-dependent potential such as 
\begin{equation*}
V(x) = - \, |x|^{2\ell}\,, \quad \ell \in \{1,2\}\,.
\end{equation*}
In consequence, the associated unbounded linear operator $A = \Delta: D(A) \subset X \to X$ generates an analytic semigroup and the linear operator $B = V: X \to X$ is bounded.
Straightforward calculations imply that the nested commutators $\ad_A^{\, q}(B)$, $q \leq p = 6$, arising in the definition of sixth-order generalized modified splitting schemes and their local error expansions are differential operator of orders~$q$, involving partial derivatives of~$V$ up to order~$2q$. 
A standard choice for the underlying function space is 
\begin{equation*}
X = L^2(\Omega, \mathbb{R}) 
\end{equation*}
or 
\begin{equation*}
X = C(\overline{\Omega}, \mathbb{R})\,,
\end{equation*}
and hence the subspace capturing the required regularity requirements for retaining the classical order of convergence corresponds to a Sobolev space that allows~$q$ weak derivatives or to continuously differentiable functions contained in $C^q(\overline{\Omega}, \mathbb{R})$. 
For localized initial states, we may disregard the effect of additional consistency requirements due to the imposed homogeneous Dirichlet or Neumann boundary conditions.   

%In the relevant case of sixth-order generalized modified splitting methods, the nested commutator $\ad_A^{\, q}(B)$ with $q = p = 6$ is indeed the dominant operator. 

\begin{theorem}
\label{thm:TheoremConvergence}
%Consider the initial value problem for the evolution equation~\eqref{eq:ode} under the basic assumptions that the unbounded linear operator $A: D(A) \subset X \to X$ generates an analytic semigroup on the underlying Banach space or a strongly continuous unitary group on a Hilbert space, respectively, and that the linear operator $B: X \to X$ is bounded.
%Apply a generalized modified operator splitting method~\eqref{eq:splitting_gen} that satisfies the classical conditions for order $p \geq 1$. 
%For an equation of parabolic type, suppose in addition that the coefficients $(a_j,b_j)_{j=1}^{s}$ are non-negative.
%Let the subspace $D \subset X$ be defined by the requirement that the  principal iterated commutator arising in the definition of the scheme or its local error expansion, respectively, is bounded
%\begin{equation*}
%\big\|\ad_A^{\, q}(B) \, e^{t (A+B)}\big\|_{X \leftarrow D} \leq \text{Const}\,, \quad 0 \leq t \leq T\,.
%\end{equation*}
%Then, the global error estimate 
%\begin{equation*}
%\|u_n - u(t_n)\|_X \leq \text{Const} \, h^p\,, \quad
%n \in \{1, \dots, N\}\,,  
%\end{equation*}
%holds with a generic constant that does not depend on the timestep and the number of time steps. 
\begin{enumerate}[(i)]
\item 
Consider the parabolic initial-boundary value problem~\eqref{eq:pde}, rewritten as an abstract evolution equation~\eqref{eq:ode} on a Banach space~$X$. 
Apply a generalized modified operator splitting method~\eqref{eq:splitting_gen} with non-negative principal coefficients $(a_j,b_j)_{j=1}^{s}$ fulfilling the classical conditions for order $p = 6$. 
Assume that the initial state and hence the exact solution to~\eqref{eq:ode} satisfy suitable regularity and consistency requirements such that 
\begin{equation*}
\big\|\ad_A^{\, q}(B) \, u(t)\big\|_X \leq \text{Const}\,, \quad 0 \leq t \leq T\,, \quad q \leq p = 6\,.
\end{equation*}
Then, the global error estimate
\begin{equation*}
\|u_n - u(t_n)\|_X \leq \text{Const} \, h^p\,, \quad n \in \{1, \dots, N\}\,,
\end{equation*}
holds with a generic constant that does not depend on the time increment and the number of time steps. 
\item 
An analogous result is valid for generalized modified operator splitting methods applied to linear Schrödinger equations 
\begin{equation*}
\text{i} \, \partial_t u = \Delta u + V(x) \, u\,, \quad (x,t) \in \Omega \times (0, T)\,, 
\end{equation*}
with unbounded operator $A = - \, \text{i} \Delta$ generating a strongly continuous unitary group on a  Hilbert space. 
Hereby, no non-negativity condition is imposed on the coefficients $(a_j,b_j)_{j=1}^{s}$.
\end{enumerate} 
\end{theorem}

\begin{proof}
In essence, the derivation of Theorem~\ref{thm:TheoremConvergence} 
follows the approach explored in our previous contributions on standard splitting methods~\cite{BlanesCasasGonzalezThalhammer2025,  ThalhammerHEO08,ThalhammerCAH12}. 
Alternatively, in our recent contribution~\cite{Thalhammer2026}, expansions of the exact and numerical solutions rely on the concept of the associated defect.  
Consistently, however, it is crucial to avoid infinite series representations of the evolution operator~$e^{t A}$ and negative times in connection with equations of parabolic type. 
The formalism of the present work 
provides the core ingredients, namely suitable local error expansions for methods of the more general format~\eqref{eq:splitting_gen}, enabling a systematic treatment of high orders and stage numbers. 
In order to avoid a repetition of the arguments, we omit the details of the proof and in particular refer to~\cite{ThalhammerHEO08,Thalhammer2026}. 
We note that a rigorous stability analysis of generalized modified splitting methods would rely on the concept of conditional stability estimates, which has been exploited in the context of standard splitting methods for nonlinear equations, see~\cite{Lubich2008}.
\end{proof}

\begin{figure}[!ht]
\centering
\includegraphics[height=0.36\textwidth]{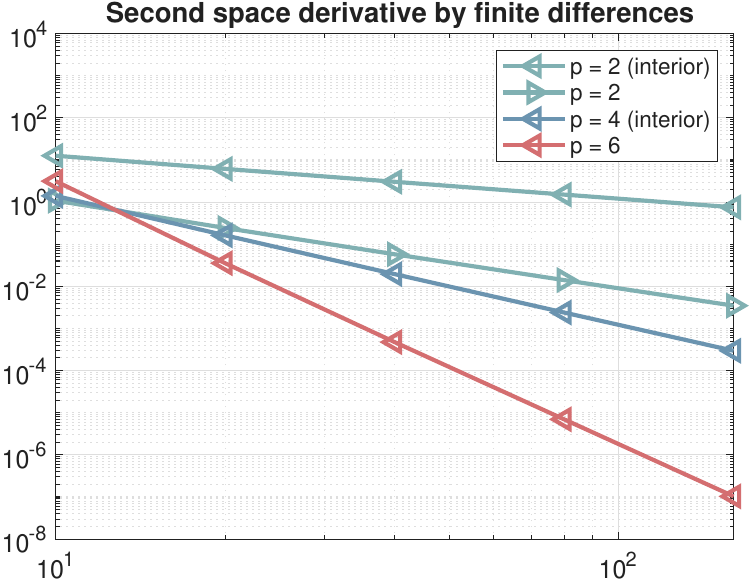} \quad 
\includegraphics[height=0.36\textwidth]{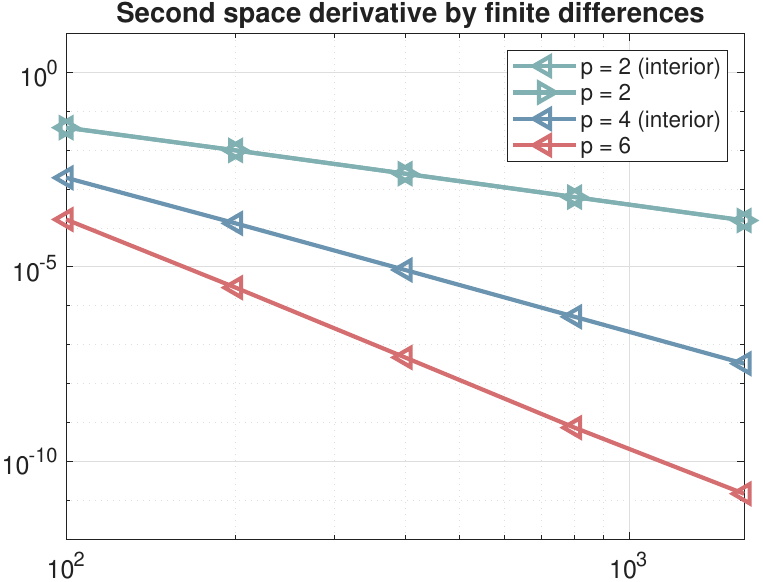}
\caption{Order confirmation of finite difference methods for polynomial (left) and Gaussian-like (right) functions. The figure was produced using an elementary Matlab script\textsuperscript{(a)}.}
\label{Figure1}    
\small\textsuperscript{(a)} \texttt{https://doi.org/10.5281/zenodo.21976559}, \cite{python_scripts}.
\end{figure}

\begin{remark*}
\begin{enumerate}[(i)]
\item 
It is remarkable that straightforward expansions of $e^{t A}$ would imply unnaturally strong regularity requirements on the prescribed initial state. 
For this reason, the representations of the exact and numerical solutions, provided in this work, avoid the presence of powers of the unbounded operator~$A$. 
Instead, series expansions are performed such that nested commutators of $A$ and $B$ arise. 
\item 
With regard to the numerical experiments given in~\cite{Thalhammer2026} for a fourth-order modified splitting method applied to linear and nonlinear evolution equations of parabolic and Schrödinger type, it is expected that the nested commutator $\ad_A^{\, p}(B)$ reflects the optimal regularity requirements, also in the context of generalized modified splittings. 
Furthermore, for initial states with reduced regularity, order reductions are encountered. 
\item 
For the sixth-order schemes constructed in  Sections~\ref{sec:high_order_chin} and~\ref{sec:high_order_gauss}, 
the principal commutators arising in the definition of $\tilde{B}$ are related to $\ad_A^{\, q}(B)$ with $q = 4$ or $q = 3$, respectively. 
Hence, adjusting the domain $D$ accordingly to the dominant nested commutator $\ad_A^{\, p}(B)$ with $p = 6$ will also ensure well-definedness of the numerical approximations. 
\end{enumerate}
\end{remark*}

\subsection{Numerical setup}\label{sec:numerical_setup}

In the following, we present numerical illustrations that confirm and complement our analysis. 
The considered test problem involves the Laplacian and a real-valued potential acting as multiplication operator
\begin{equation*}
A = c \Delta\,, \quad B = c V\,.
\end{equation*}
Positive constants $c > 0$ correspond to parabolic equations, whereas purely imaginary constants lead to equations of Schrödinger type. 
In order to ensure sufficient regularity of the exact solution, we prescribe a Gaussian-like initial state on a bounded domain. 
We employ a spatial discretization by finite differences of orders 2, 4, and 6, respectively, see Figure~\ref{Figure1}. 
For simplicity, we restrict ourselves to a single space dimension. 
Due to the fact that standard central finite differences can be realized by Fast Fourier transforms and with regard to our previous numerical experiments for standard and modified splitting methods, we expect that the same conclusions are valid in higher space dimensions. 
The temporal discretization relies on the two sixth-order generalized modified splitting methods introduced in Sections~\ref{sec:high_order_chin} and~\ref{sec:high_order_gauss}.
The construction and implementation of other high-order schemes would follow the same lines. 

\begin{figure}[!ht]
\centering
\includegraphics[width=0.45\textwidth]{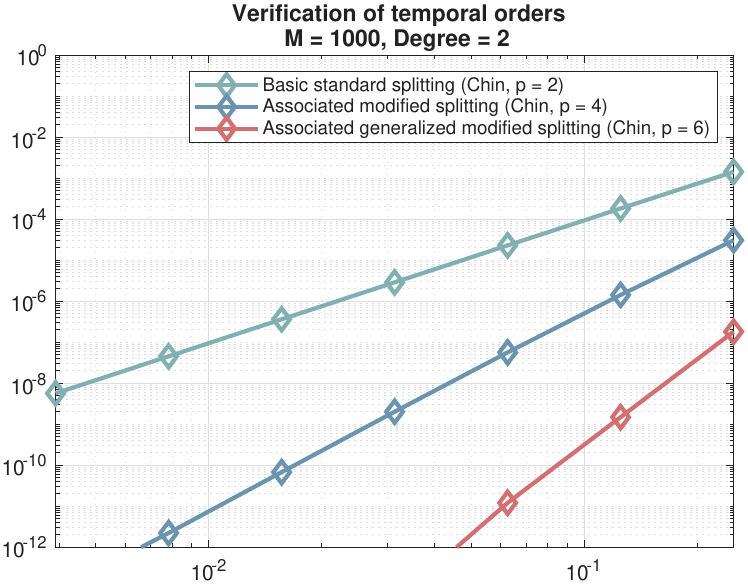} \quad
\includegraphics[width=0.45\textwidth]{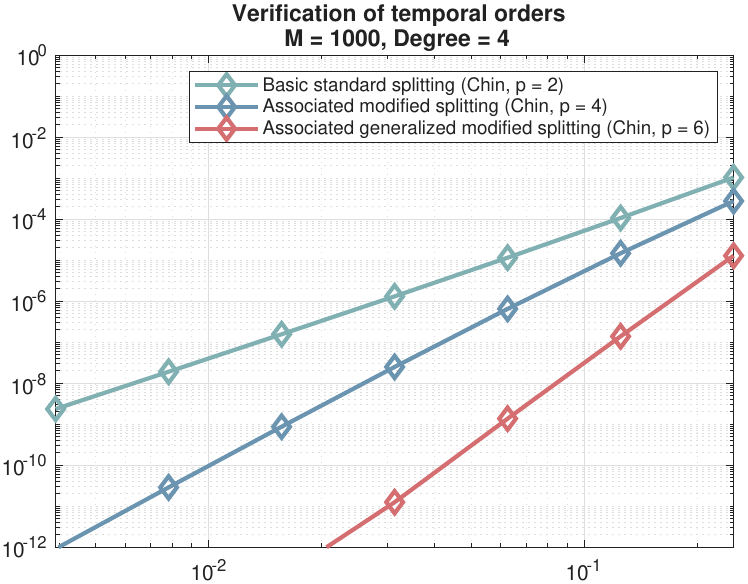} \\[2mm]
\includegraphics[width=0.45\textwidth]{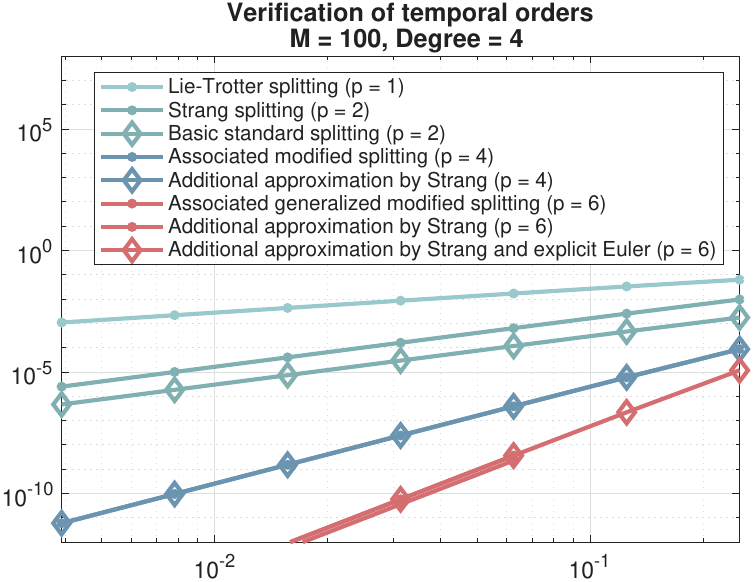} \quad 
\includegraphics[width=0.45\textwidth]{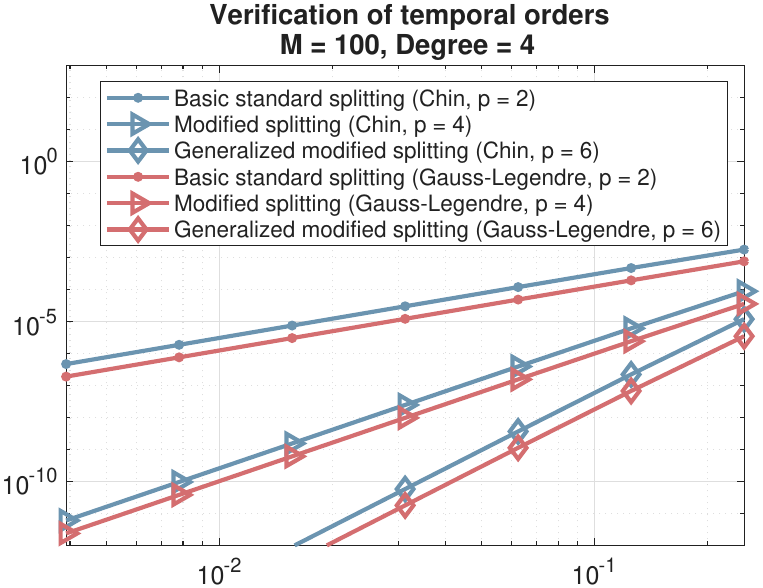}
\caption{Application of generalized modified splitting methods to spatially semi-discretized systems defined by two square matrices of dimension~$M$.
The underlying heat-type equation involves a quadratic (degree 2) or quartic (degree 4) potential, respectively. 
First row: Local errors versus time increments.  
Verification of the local temporal order $p + 1 = 7$ for the extension of Chin's modified potential method (Section~\ref{sec:high_order_chin}). 
Second row: Global errors versus time increments. 
Verification of the global temporal orders $p = 6$ for generalized modified splitting methods.
Left: Comparison of different approaches and observation of a stabilising effect for an additional approximation by Strang and explicit Euler. 
Right: Comparison with an alternative generalized modified splitting method based on Gauss--Legendre nodes (Section~\ref{sec:high_order_gauss}).
The figure was produced using an elementary Matlab script\textsuperscript{(a)}.}
\label{Figure2}    
\small\textsuperscript{(a)} \texttt{https://doi.org/10.5281/zenodo.21976559}, \cite{python_scripts}.
\end{figure}

\subsection{Results and comparisons}\label{sec:numerical_results}

In all cases, the classical orders are retained, see Figures~\ref{Figure2} and~\ref{Figure3}.
A separation of the term involving the highest power of the timestep by means of a second-order splitting and additional approximation by the explicit Euler method has a stabilizing effect. 

\begin{figure}[!ht]
\centering
\includegraphics[width=0.45\textwidth]{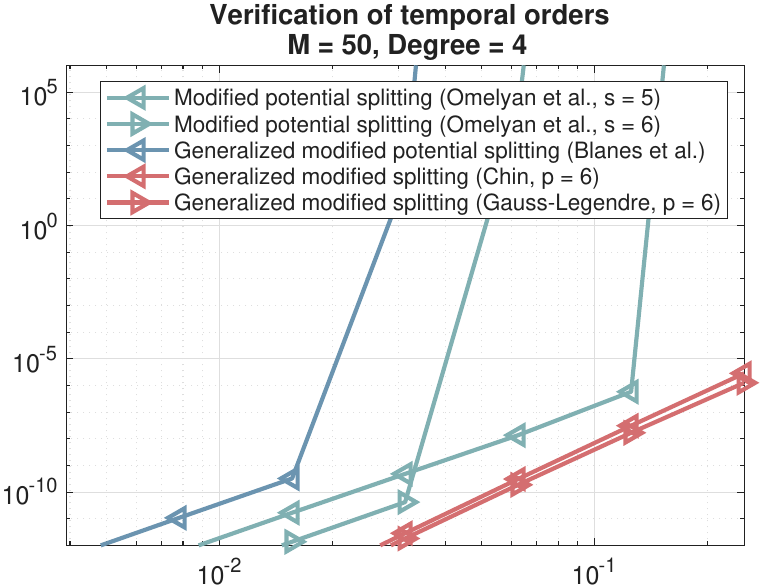} \quad 
\includegraphics[width=0.45\textwidth]{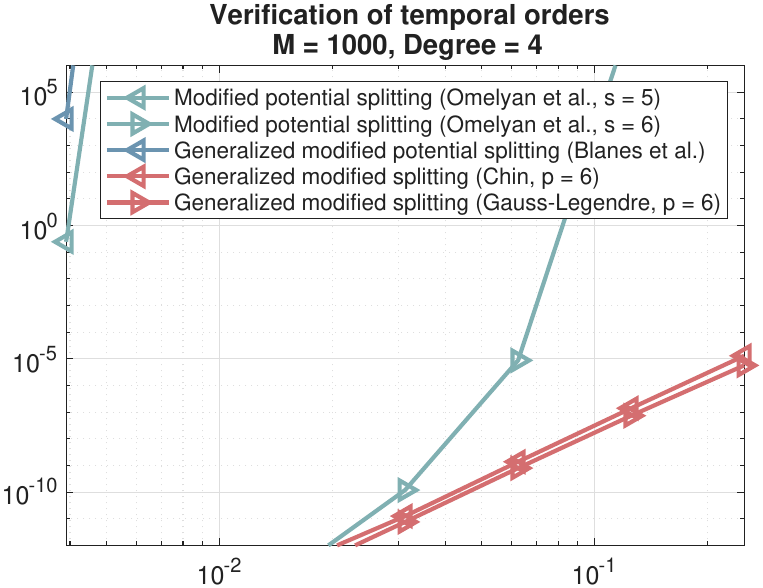}
\caption{Application of sixth-order splitting-type methods to spatially semi-discretized systems defined by two square matrices of dimension~$M$
(Sections~\ref{sec:high_order_chin} and~\ref{sec:high_order_gauss}, and Appendix~\ref{app:literature_methods}).
Representations of the local errors versus the time increments. 
As already illustrated in Figure~\ref{Figure2}, the extension of Chin's modified potential method (Section~\ref{sec:high_order_chin}) remains stable and retains the local temporal order $p + 1 = 7$, whereas 
the modified potential splitting methods introduced in~\cite{OmelyanCHF02} involve negative coefficients and hence suffer from instabilities for systems of higher dimension (left: $M = 50$, right: $M = 1000$) and larger time increments. 
Likewise, severe instabilities occur for the processed generalized modified potential splitting constructed in~\cite{BlanesHRK01}. 
For the employed matrix-based implementation of finite difference approximations, the commutator $[B,[B,[A,B]]]$ does not vanish, wherefore the reduced local orders $p + 1 = 5$ are observed for the modified potential and Runge--Kutta--Nyström methods, respectively.}
\label{Figure3}    
\small\textsuperscript{(a)} \texttt{https://doi.org/10.5281/zenodo.21976559}, \cite{python_scripts}.
\end{figure}

\FloatBarrier

\section{Conclusions}\label{sec:conclusions}

We have developed a framework for the analysis and construction of generalized modified splitting methods that retains the evolution generated by $A$ and organizes the remaining terms through nested commutators. The composition law tracks elapsed time and gives explicit coefficients for combining methods with independently scaled timesteps. The stage-dependent substitution law describes how corrections inside individual exponentials change the coefficients of the complete method. Together, these laws provide complementary tools for analyzing compositions and constructing modified stages without expanding the distinguished evolution $e^{hA}$.

The substitution law yields a constructive procedure for canceling the leading error by modifying a single stage, while leaving the principal splitting coefficients and the number of stages unchanged. Under the stated symmetry assumptions, a central-stage correction increases the formal order by two. The Chin-based and Gauss-node constructions illustrate this procedure at order six with nonnegative principal coefficients. They also illustrate a trade-off between the commutator corrections required and the number of principal substeps. The accompanying scripts support coefficient generation, order verification, and further applications of the construction.

For evolution equations with unbounded operators, formal coefficient cancellation must be complemented by regularity, well-definedness, and stability estimates. The convergence discussion and numerical experiments address this analytical motivation and illustrate the behavior of the constructed methods for the spatial discretizations considered. Nonnegative principal coefficients avoid backward substeps along the original generators and provide better stability properties of the constructed methods when applied to parabolic problems.

Further work includes the construction of optimal generalized modified splitting methods as well as the extension of the formal calculus to nonlinear evolution equations through Lie derivatives together with a corresponding analysis of domains, regularity, and stability.

\section*{Acknowledgments}

The authors would like to thank Adrien Busnot Laurent for helpful discussions. EB has been supported by the Knut and Alice Wallenberg Foundation (grant number KAW 2023.0433). MT has been supported by the Austrian Science Fund FWF through a stand-alone project (grant-doi 10.55776/PAT1281625).

\appendix

\section{Order $6$ conditions for symmetric generalized modified splitting methods}
\label{app:order_conditions}

We apply Theorem~\ref{thm:substitution} to compute the order conditions for symmetric generalized modified splitting methods. A generalized modified splitting method has order $p$ if $(\beta \star \alpha) (\mu) = (1/\gamma)(\mu)$ for $\mu \in \N^\infty, |\mu| \leq p$ where $B_j = \frac{1}{h}B(\beta_j)$ in \eqref{eq:splitting_gen}. Due to the discussion in Section \ref{sec:lyndon}, it suffices to consider the order conditions corresponding to Lyndon multi-indices $\mu \in \N^\infty_L$ with $|\mu| \leq p$.

\begin{center}
    \textbf{Assumptions}
\end{center}
\[ \sum_{i=1}^s a_i = 1 \,, \quad \beta_j((0)) = 1 \,, \quad \text{for } j=1,\dots,s \,. \]

\begin{center}
    \textbf{Order 1, 2}
\end{center}
\[ (\beta \star \alpha)((0)) = \alpha((0)) = 1 \,. \]

\begingroup
\allowdisplaybreaks

\begin{center}
    \textbf{Order 3, 4}
\end{center}
\begin{align*}
    (\beta \star \alpha)((0,1))
        &= \sum_{\lambda \in \Lambda_1}^! \beta_{\lambda_1}((0,1)) \alpha((0), \lambda) \\
        &\quad + \sum_{\lambda \in \Lambda_2}^! \beta_{\lambda_2}((1)) \alpha((0,0), \lambda)
        + \alpha((0,1)) = \frac{1}{3} \,, \\
    (\beta \star \alpha)((2))
        &= \sum_{\lambda \in \Lambda_1}^! \beta_{\lambda_1}((2)) \alpha((0), \lambda) \\
        &\quad + 2 \sum_{\lambda \in \Lambda_1}^! \beta_{\lambda_1}((1)) \alpha((1), \lambda)
        + \alpha((2)) = \frac{1}{3} \,.
\end{align*}

\begin{center}
    \textbf{Order 5, 6}
\end{center}
\begin{align*}
    (\beta \star \alpha)((0,0,0,1))
        &= \sum_{\lambda \in \Lambda_1}^! \beta_{\lambda_1}((0,0,0,1)) \alpha((0), \lambda) \\
        &\quad + \sum_{\lambda \in \Lambda_2}^! \beta_{\lambda_2}((0,0,1)) \alpha((0,0), \lambda) \\
        &\quad + \sum_{\lambda \in \Lambda_3}^! \beta_{\lambda_3}((0,1)) \alpha((0,0,0), \lambda) \\
        &\quad + \sum_{\lambda \in \Lambda_4}^! \beta_{\lambda_4}((1)) \alpha((0,0,0,0), \lambda)
        + \alpha((0,0,0,1)) = \frac{1}{30} \,, \\
    (\beta \star \alpha)((0,0,2))
        &= \sum_{\lambda \in \Lambda_1}^! \beta_{\lambda_1}((0,0,2)) \alpha((0), \lambda)
        + 2 \sum_{\lambda \in \Lambda_1}^! \beta_{\lambda_1}((0,0,1)) \alpha((1), \lambda) \\
        &\quad + \sum_{\lambda \in \Lambda_2}^! \beta_{\lambda_2}((0,2)) \alpha((0,0), \lambda)
        + 2 \sum_{\lambda \in \Lambda_2}^! \beta_{\lambda_2}((0,1)) \alpha((0,1), \lambda) \\
        &\quad + \sum_{\lambda \in \Lambda_3}^! \beta_{\lambda_3}((2)) \alpha((0,0,0), \lambda)
        + 2 \sum_{\lambda \in \Lambda_3}^! \beta_{\lambda_3}((1)) \alpha((0,0,1), \lambda) \\
        &\quad + \alpha((0,0,2)) = \frac{1}{10} \,, \\
    (\beta \star \alpha)((0,1,1))
        &= \sum_{\lambda \in \Lambda_1}^! \beta_{\lambda_1}((0,1,1)) \alpha((0), \lambda) - \sum_{\lambda \in \Lambda_1}^! \beta_{\lambda_1}((0,0,1)) \alpha((1), \lambda) \\
        &\quad + \sum_{\lambda \in \Lambda_2}^! \beta_{\lambda_1}((0,1)) \beta_{\lambda_2}((1)) \alpha((0,0), \lambda)
        + \sum_{\lambda \in \Lambda_2}^! \beta_{\lambda_1}((0,1)) \alpha((0,1), \lambda) \\
        &\quad + \sum_{\lambda \in \Lambda_3}^! \beta_{\lambda_2}((1)) \beta_{\lambda_3}((1)) \alpha((0,0,0), \lambda) \\
        &\quad + \sum_{\lambda \in \Lambda_3}^! \beta_{\lambda_3}((1)) \alpha((0,1,0), \lambda)
        + \sum_{\lambda \in \Lambda_3}^! \beta_{\lambda_2}((1)) \alpha((0,0,1), \lambda) \\
        &\quad + \alpha((0,1,1)) = \frac{1}{15} \,, \\
    (\beta \star \alpha)((0,3))
        &= \sum_{\lambda \in \Lambda_1}^! \beta_{\lambda_1}((0,3)) \alpha((0), \lambda)
        + 3 \sum_{\lambda \in \Lambda_1}^! \beta_{\lambda_1}((0,2)) \alpha((1), \lambda) \\
        &\quad + 3 \sum_{\lambda \in \Lambda_1}^! \beta_{\lambda_1}((0,1)) \alpha((2), (\lambda_1))
        + \sum_{\lambda \in \Lambda_2}^! \beta_{\lambda_2}((3)) \alpha((0,0), \lambda) \\
        &\quad + 3 \sum_{\lambda \in \Lambda_2}^! \beta_{\lambda_2}((2)) \alpha((0,1), \lambda)
        + 3 \sum_{\lambda \in \Lambda_2}^! \beta_{\lambda_2}((1)) \alpha((0,2), \lambda) \\
        &\quad + \alpha((0,3)) = \frac{1}{5} \,, \\
    (\beta \star \alpha)((1,2))
        &= \sum_{\lambda \in \Lambda_1}^! \beta_{\lambda_1}((1,2)) \alpha((0), \lambda)
        + \sum_{\lambda \in \Lambda_1}^! \beta_{\lambda_1}((0,2)) \alpha((1), \lambda) \\
        &\quad + \sum_{\lambda \in \Lambda_1}^! \beta_{\lambda_1}((0,1)) \alpha((2), \lambda)
        + \sum_{\lambda \in \Lambda_2}^! \beta_{\lambda_1}((1)) \beta_{\lambda_2}((2)) \alpha((0,0), \lambda) \\
        &\quad + \sum_{\lambda \in \Lambda_2}^! \beta_{\lambda_2}((2)) \alpha((1,0), \lambda)
        + 2 \sum_{\lambda \in \Lambda_2}^! \beta_{\lambda_1}((1)) \beta_{\lambda_2}((1)) \alpha((0,1), \lambda) \\
        &\quad + 2 \sum_{\lambda \in \Lambda_2}^! \beta_{\lambda_2}((1)) \alpha((1,1), \lambda)
        + \sum_{\lambda \in \Lambda_2}^! \beta_{\lambda_1}((1)) \alpha((0,2), \lambda) \\
        &\quad + \alpha((1,2)) = \frac{1}{10} \,, \\
    (\beta \star \alpha)((4))
        &= \sum_{\lambda \in \Lambda_1}^! \beta_{\lambda_1}((4)) \alpha((0), \lambda)
        + 4 \sum_{\lambda \in \Lambda_1}^! \beta_{\lambda_1}((3)) \alpha((1), \lambda) \\
        &\quad + 6 \sum_{\lambda \in \Lambda_1}^! \beta_{\lambda_1}((2)) \alpha((2), \lambda)
        + 4 \sum_{\lambda \in \Lambda_1}^! \beta_{\lambda_1}((1)) \alpha((3), \lambda) \\
        &\quad + \alpha((4)) = \frac{1}{5} \,.
\end{align*}

\endgroup

\section{Representative sixth-order methods from the literature}
\label{app:literature_methods}

This appendix records the methods discussed in Section~\ref{sec:literature} and used for comparison in Section~\ref{SectionConvergence}, in the notation of \eqref{eq:splitting_gen}. The stated orders refer to the operator structures considered in the cited works; rewriting the methods in this notation does not establish their order for arbitrary operators $A,B$.

\subsection{Force-gradient methods of Omelyan, Mryglod, and Folk}
\label{app:omelyan_methods}

We select two of the sixth-order force-gradient methods in \cite[Section II.B.5]{OmelyanCHF02}. Both use only the correction $[B,[A,B]]$ and have negative principal coefficients.

\paragraph{Method with $s=5$ stages.} The method corresponding to equations (45) and (47) of \cite{OmelyanCHF02} is given by
\begin{align*}
    a_1 &= 0 \,, \\
    b_1 = b_5 &= 0.35995080879414365 \,, \\
    a_2 = a_5 &= 1.079852426382431 \,, \\
    b_2 = b_4 &= -0.14371472730265417 \,, \\
    a_3 = a_4 &= -0.579852426382431 \,, \\
    b_3 & = 0.567527837017021 \,,
\end{align*}
with operators $B_j$ given by
\begin{align*}
    B_1 &= B_5 = B \,, \\
    B_2 &= B_4 = B + 0.0971734385636685 \  h^2 [B, [A, B]] \,, \\
    B_3 &= B - 0.06915436886520256 \  h^2 [B, [A, B]] \,.
\end{align*}

\paragraph{Method with $s=6$ stages.} For the method corresponding to equations (50) and (51) of \cite{OmelyanCHF02} with optimal leading error, the coefficients are
\begin{align*}
    a_1 = a_6 &= 0.1097059723948682 \,, \\
    b_1 = b_5 &= 0.2693315848935301 \,, \\
    a_2 = a_5 &= 0.4140632267310831 \,, \\
    b_2 = b_4 &= 1.131980348651556 \,, \\
    a_3 = a_4 &= -0.023769199125951324 \,, \\
    b_3 & = -1.8026238670901722 \,, \\
    b_6 &= 0 \,,
\end{align*}
with operators $B_j$ given by
\begin{align*}
    B_1 &= B_5 = B + 0.0032087440999994535 \ h^2 [B, [A, B]] \,, \\
    B_2 &= B_4 = B \,, \\
    B_3 &= B + 0.0073483917948690405 \ h^2 [B, [A, B]]\,.
\end{align*}

\subsection{A processed method of Blanes, Casas, and Ros}
\label{app:blanes_method}

The following sixth-order processed method is drawn from the family in \cite[Section 3.1]{BlanesHRK01}. Its central modification includes degree-seven commutators to reduce the leading error within that family. For an invertible processor $P_h$ and kernel $K_h$, write the processed map as
\begin{equation}
    \label{eq:literature_processing}
    \Phi_h=P_h^{-1}K_hP_h \,, \qquad
    \Phi_h^N=P_h^{-1}K_h^NP_h \,.
\end{equation}
For a fixed timestep, the kernel is applied at every step, whereas processing is required at initialization and when recovering the physical variables for output.

The displayed composition includes both the kernel and the processing factors. The factors associated with $b_1,a_1,\ldots,b_6,a_6$ and $b_{10},a_{11},\ldots,b_{15},a_{16}$ form the two inverse processing blocks; the other factors form the kernel. For a fixed timestep, adjacent inverse processing blocks cancel when composing steps, as in \eqref{eq:literature_processing}. Thus the number of factors in the full processed map should not be interpreted as the recurring cost per step.

With $b_{16}=0$, the coefficients are
\begin{equation*}
    \begin{aligned}
        a_1 = -a_{16} &= -0.1604630501234888 \,, \\
        b_1 = -b_{15} &= 0.012334538446142270 \,, \\
        a_2 = -a_{15} &= 0.1222126706298830 \,, \\
        b_2 = -b_{14} &= 0.6610294848488182 \,, \\
        a_3 = -a_{14} &= -0.1916801124727711 \,, \\
        b_3 = -b_{13} &= 0.023112349678219939 \,, \\
        a_4 = -a_{13} &= -0.5630722377955035 \,, \\
        b_4 = -b_{12} &= -0.000181521815949959 \,, \\
    \end{aligned}
    \qquad
    \begin{aligned}
        a_5 = -a_{12} &= 0.7612758792358986 \,, \\
        b_5 = -b_{11} &= -2.3768244683666757 \,, \\
        a_6 = -a_{11} &= -a_1 - a_2 - a_3 - a_4 - a_5 \,, \\
        b_6 = -b_{10} &= -b_1 - b_2 - b_3 - b_4 - b_5 \,, \\
        a_7 = a_{10} &= -0.0682610383918630 \,, \\
        b_7 = b_9 &= 0.2621129352517028 \,, \\
        a_8 = a_9 &= 0.568261038391863 \,, \\
        b_8 &= 0.47577412949659437 \,, \\
    \end{aligned}
\end{equation*}
with operators $B_j$ given by
\begin{align*}
    B_1 = B_{15} &= B - 1.1201212144227515 \ h^2 [B, [A, B]] \,, \\
    B_2 = B_{14} &= B + 0.07607581926383974 \ h^2 [B, [A, B]] \,, \\
    B_3 = B_{13} &= B + 0.5824764836072831 \ h^2 [B, [A, B]] \,, \\
    B_4 = B_{12} &= B + 3.326427681441248 \ h^2 [B, [A, B]] \,, \\
    B_5 = B_{11} &= B - 0.004207294284071334 \ h^2 [B, [A, B]] \,, \\
    B_6 = B_{10} &= B - 0.01 / b_6 \ h^2 [B, [A, B]] \,, \\
    B_7 = B_{9} &= B \,, \\
    B_8 &= B + 0.034472477167751724 \ h^2 [B, [A, B]] \\
        &\qquad + 3.9135085510111535 \times 10^{-5} \ h^6 [B, [A, [B, [B, [A, [A, B]]]]]] \\
        &\qquad - 1.3274323034002436 \times 10^{-5} \ h^6 [B, [B, [B, [A, [A, [A, B]]]]]] \,.
\end{align*}

\bibliographystyle{emss}
\bibliography{references}

\end{document}